\documentclass[letterpaper,10pt]{article}
\usepackage[textwidth=125mm,textheight=195mm]{geometry}
\usepackage{amsmath,amssymb,amsthm,pinlabel,tikz}

\newcommand{\nc}{\newcommand}
\nc{\dmo}{\DeclareMathOperator}
\dmo{\ra}{\rightarrow}
\dmo{\N}{\mathbb{N}}
\dmo{\Z}{\mathbb{Z}}
\dmo{\R}{\mathbb{R}}
\dmo{\C}{\mathcal{C}}
\dmo{\AC}{\mathcal{AC}}
\dmo{\Mod}{Mod}
\nc{\nt}{\newtheorem}

\nt{theorem}{Theorem}
\nt{lemma}{Lemma}

\newtheorem{thm}{{\bf Theorem}}[section]
\newtheorem{lem}[thm]{{\bf Lemma}}

\newtheorem{corollary}[theorem]{{\bf Corollary}}
\newtheorem{prop}[thm]{{\bf Proposition}}

\numberwithin{equation}{section}

\begin{document}

\title{Small asymptotic translation lengths of pseudo-Anosov maps on the curve complex}

\author{Eiko Kin\\
  \small Department of Mathematics, Osaka University Toyonaka\\
  \small \texttt{kin@math.sci.osaka-u.ac.jp}
~\\  
~\\
  Hyunshik Shin\\
  \small Department of Mathematical Sciences, KAIST,\\
  \small \texttt{hshin@kaist.ac.kr}
}

%\date{\today}
\date{}

\newcommand{\Addresses}{{
  \bigskip
  \footnotesize

 \textsc{Department of Mathematics, Graduate School of Science, Osaka University Toyonaka,
                Osaka 560-0043, JAPAN}\par\nopagebreak
  \textit{E-mail address}: \texttt{kin@math.sci.osaka-u.ac.jp}

  \medskip

  \textsc{Department of Mathematical Sciences, KAIST,
  						291 Daehak-ro Yuseong-gu, Daejeon, 34141, South Korea }\par\nopagebreak
  \textit{E-mail address}: \texttt{hshin@kaist.ac.kr}

}}

%%%%%%%%%%%%%%%%%%%%%%%%%%%%%%%%%%%%%%%%%%%%%%%%%%%%%%%%%%%%%
%																							
%						Abstract															%
%%%%%%%%%%%%%%%%%%%%%%%%%%%%%%%%%%%%%%%%%%%%%%%%%%%%%%%%%%%%%

\maketitle

\begin{abstract}
\noindent
Let $M$ be a hyperbolic fibered 3-manifold whose first Betti number is greater than 1 and let $S$ be
a fiber with pseudo-Anosov monodromy $\psi$. We show that there exists a 
sequence $(R_n, \psi_n)$ of fibers and monodromies contained in the fibered cone of $(S,\psi)$
such that the asymptotic translation length of $\psi_n$ on the curve complex 
$\C(R_n)$ behaves asymptotically like $1/|\chi(R_n)|^2$.
As applications, we can reprove the previous result by Gadre--Tsai
that the minimal asymptotic translation length of a closed surface of genus $g$ 
asymptotically behaves like $1/g^2$.
We also show that this holds for the cases of hyperelliptic mapping class group and
hyperelliptic handlebody group.
\end{abstract}

\vspace{1em}

\vspace{1em}

%%%%%%%%%%%%%%%%%%%%%%%%%%%%%%%%%%%%%%%%%%%%%%%%%%%%%%%%%%%%%
%
%									Introduction								
%
%%%%%%%%%%%%%%%%%%%%%%%%%%%%%%%%%%%%%%%%%%%%%%%%%%%%%%%%%%%%%

\section{Introduction}
\label{section_introduction}

Let $S_{g,n}$ be an orientable surface of genus $g$ with $n$ punctures.
We will simply denote it by $S$.
The \textit{mapping class group} of $S$, denoted  $\Mod(S)$, is the group
of isotopy classes of orientation-preserving homeomorphisms of $S$.
By the Nielsen--Thurston classification theorem, each element of $\Mod(S)$
is either periodic, reducible, or pseudo-Anosov.

For a non-sporadic surface $S$, that is, a surface with $3g-3+n \geq 2$, 
the {\it curve complex} $\C(S)$ is defined to be a simplicial complex whose vertex set $\mathcal{C}^{0}(S)$ 
is the set of homotopy classes
of essential simple closed curves in $S$, and whose $k$-simplices are formed by $k+1$
distinct vertices whose representatives can be chosen to be pairwise disjoint.
We will restrict our attention to the $1$-skeleton $\mathcal{C}^{1}(S)$ of $\C(S)$ with path metric $d_{\C}$
by assigning each edge length $1$. 
Then $\Mod(S)$ acts on $\C(S)$ by isometry and the \textit{asymptotic translation length}
of $f \in \Mod(S)$ on $\C^{1}(S)$ is defined by 
$$\ell_{\C}(f) = \liminf_{j \ra \infty} \frac{d_{\C}(\alpha, f^j(\alpha))}{j},$$
where $\alpha$ is an essential simple closed curve in $S$. It follows from the definition that
$\ell_{\C}(f)$ is independent of the choice of $\alpha$ and that $\ell_{\C}(f^k) = k \hspace{0.3mm} \ell_{\C}(f)$ for $k \in \N$.

Masur and Minsky \cite{MasurMinsky99} showed that $f \in \Mod(S)$ is pseudo-Anosov if
and only if $\ell_{\C}(f)>0$, and Bowditch \cite{Bowditch08} proved that there exists 
a constant $m>0$ depending only on the surface $S$ such that for each pseudo-Anosov mapping
class $f$ in $\Mod(S)$, $f^k$ has an invariant geodesic axis on $\C(S)$ for some $k \leq m$. 
In other words, 	
$\ell_{\C}(f)$ is a positive rational number with bounded denominator.

For any subgroup $H < \Mod(S)$, let us denote by $L_{\C}(H)$ the minimum of
$\ell_{\C}(f)$ over all pseudo-Anosov elements $f \in H$. 
Then  $L_{\mathcal{C}}(H) \ge L_{\mathcal{C}}( \Mod(S))$. 
We also write $F \asymp G$ if there exists a universal constant $C>0$ so that
$ 1/C \leq F/G \leq C$. 
For the closed surface $S_g$ of genus $g$, Gadre and Tsai \cite{GadreTsai11} showed that
$$L_{\C}(\Mod(S_g)) \asymp \frac{1}{g^2}.$$
In fact, using the invariant train tracks constructed by Bestvina and Handel \cite{BestvinaHandel95} and
the nesting lemma by Masur and Minsky \cite{MasurMinsky99},
they obtained  in \cite{GadreTsai11} 
the lower bound of the asymptotic translation lengths 
in terms of the Euler characteristic $\chi(S_{g,n})$ of $S_{g,n}$. That is,
$$\ell_{\C}(f) \geq \frac{1}{18\chi(S_{g,n})^2+ 30 |\chi(S_{g,n})|-10n}$$ 
for any pseudo-Anosov element $f \in \Mod(S_{g,n})$. 
To obtain the upper bound, 
they use an explicit family of pseudo-Anosov mapping classes.
This family was first considered by Penner \cite{Penner91} to find small stretch factors of pseudo-Anosov
maps;
$$L_{\mathcal{C}}(\mathrm{Mod}(S_g)) \leq \frac{4}{g^2+g-4}.$$

In this paper, we describe a way of generating a sequence of pseudo-Anosov mapping classes
$\psi_n \in \Mod(S_n)$ with small asymptotic translation lengths on the curve complex.
We say that a sequence $\{ \psi_n \}$ has a \textit{small} asymptotic translation length
if $\ell_{\C}(\psi_n) \asymp 1/|\chi(S_n)|^2$,
where $\chi(S_n)$ is the Euler characteristic of the corresponding surface $S_n$ such that $|\chi(S_n)| \ra \infty$ as $n \ra \infty$.

Let $M$ be a hyperbolic fibered 3-manifold with the first Betti number $b_1(M) \geq 2$ and let 
$S \subset M$ be a fiber with pseudo-Anosov monodromy $\psi$. 
Then the assumption $b_1(M) \geq 2$ implies that there is a primitive cohomology
class $\xi_0 \in H^1(S;\Z)$ fixed by $\psi$, that is,
$\xi_0 \circ \psi_* = \xi_0$, where 
$\psi_*: H_1(S; {\Bbb Z}) \rightarrow  H_1(S; {\Bbb Z}) $. 
Let $p:\widetilde{S} \ra S$ be a $\Z$-covering
map corresponding to $\xi_0$ whose deck transformation group is generated by
$h:\widetilde{S} \ra \widetilde{S}$ and let $\widetilde{\psi}$ be a lift of 
$\psi$ to $\widetilde{S}$.
Then we have the following main theorem.

\begin{theorem}\label{thm:main}
For all sufficiently large $n$, 
$R_n = \widetilde{S}/\langle h^n\widetilde{\psi} \rangle$
is a fiber of $M$ with $|\chi(R_n)| \asymp n$ 
whose 
pseudo-Anosov monodromy $\psi_n$ satisfies 
$$ \ell_{\C}(\psi_n) \asymp \frac{1}{|\chi(R_n)|^2}.$$
\end{theorem}

The above family of fibers in a fibered 3-manifold was first considered by McMullen
and he proved the following theorem providing short geodesics on the moduli space
when $S$ is a closed surface.

\begin{thm}[McMullen, Theorem~10.2 in \cite{McMullen00}] \label{thm:McMullen}
For all $n$ sufficiently large,
$$R_n = \widetilde{S}/\langle h^n\widetilde{\psi} \rangle$$
is a closed surface of genus $g_n \asymp n$, and $h^{-1}:\widetilde{S} \ra \widetilde{S}$
descends to a pseudo-Anosov mapping class $\psi_n \in \Mod(R_n)$ with
$$\log \lambda(\psi_n) \asymp \frac{1}{g_n},$$
where $\lambda(\psi_n)$ is the stretch factor of $\psi_n$.
\end{thm}
Although McMullen dealt with closed hyperbolic 3-manifolds in Theorem~\ref{thm:McMullen}, we can adopt the same proof for the
general case of fibers of cusped hyperbolic $3$-manifolds. 
In such case, we have to say
$\log \lambda(\psi_n) \asymp 1/|\chi(R_n)|$ and $|\chi(R_n)| \asymp n$. 

\begin{figure}[tp]
\labellist \small\hair 2pt
\pinlabel $a_1$ at 32 38
\pinlabel $a_2$ at 59 46
\pinlabel $a_3$ at 95 32
\pinlabel $b_1$ at 12 34
\pinlabel $b_2$ at 80 41
\pinlabel $c$ at 75 -2
\pinlabel $\alpha_1$ at 170 30
\pinlabel $\alpha_2$ at 208 29
\pinlabel $\beta_1$ at 162 8
\pinlabel $\beta_2$ at 199 8
\endlabellist
\centering
\includegraphics[scale=1.4]{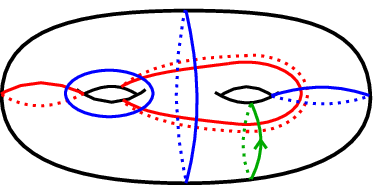}
~~~~
\includegraphics[scale=1.4]{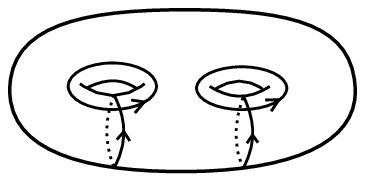}
\caption{Simple closed curves and the standard basis for $H_1(S_2;\Z)$. }
\label{figure:genus2}
\end{figure}

As a consequence of Theorem \ref{thm:main}, we can determine the behavior of minimal asymptotic translation
lengths of a few subgroups of mapping class groups.
First of all, the fact that $L_{\C}(\Mod(S_g)) \asymp 1/g^2$ also follows from
Theorem \ref{thm:main} by considering the genus 2 surface $S_2$ and any mapping class fixing a nontrivial cohomology class.
For instance, consider the mapping class 
$\psi = T_{a_1}T_{a_2}T_{a_3}T_{b_1}^{-1}T_{b_2}^{-1}$ of the closed surface of genus 2 as in 
Figure \ref{figure:genus2}, where $T_{\gamma}$ is the left-handed Dehn twist about a simple closed curve $\gamma$. 
(We apply elements of the mapping class group from right to left.) 
Then $\psi$ is pseudo-Anosov because it is coming from Penner's construction (see, for instance, \cite[Theorem 14.4]{FarbMargalit12}).
The action of $\psi$ on the first homology $H_1(S_2;\Z)$ with respect to
the basis $\{\alpha_1, \beta_1, \alpha_2, \beta_2 \}$ as in Figure \ref{figure:genus2}
is given by the matrix
$$\left(
\begin{array}{rrrr}
 2 &-1& 0 & 0\\
-1 & 1& 0 & 0\\
 0 & 0 & 1 & 1\\
 0 & 0 & 0 & 1
\end{array}
 \right).
$$
Therefore there is a 1-dimensional subspace of $H_1(S_2;\Z)$ and its dual $\xi_0 \in H^1(S_2;\Z)$
is given by the algebraic intersection number with an oriented simple closed curve $c$.
Hence, $\xi_0$ is a cohomology class fixed by $\psi$.

{\renewcommand{\arraystretch}{2}
\begin{table}[tp]
\label{table_lc}
\centering
\begin{tabular}{|l || c | c | c | }
\hline
& $\mathrm{Mod}(S_{0,n})$ & $ \mathrm{Mod}(S_{1,2n})$ & For any fixed $g \ge 2$, $\mathrm{Mod}(S_{g,n})$  \\
\hline
$L_{\mathcal{C}}( \cdot )$& $\asymp \dfrac{1}{n^2}$ ~\cite{Valdivia14} & $\asymp \dfrac{1}{n^2}$ ~\cite{GadreTsai11} &  $\asymp \dfrac{1}{n}$ ~\cite{Valdivia14} \\[4pt]
\hline
\end{tabular}
\vspace{0.5em}
\caption{Minimal asymptotic translation lengths.}
\end{table}}
Valdivia \cite{Valdivia14} showed that fixing $g \geq 2$  as $n \to \infty$, 
$$L_{\C} (\Mod(S_{g,n})) \asymp \frac{1}{n},$$
and for the remaining cases of $S_{0,n}$ and $S_{1,2n}$ with even number of punctures as $n \to \infty$,  see Table~\ref{table_lc}. 
%and for the remaining cases of $g=0$ and $1$ as $n \to \infty$,  see Table~\ref{table_lc}. 
We will determine the minimal asymptotic translation lengths of a few other types of surfaces 
including the surface of genus $1$ with odd number of punctures.
Let $D_n$ be the closed disk $D$ with $n$-punctures and 
let $\mathrm{Mod}(D_n)$ be the mapping class group of $D_n$ fixing the boundary $\partial D$ of the disk $D$ pointwise. 
As an application of Theorem~\ref{thm:main}, we have the following results.

\begin{theorem}
\label{thm:sphere-torus}
We have 
\begin{enumerate}
\item[(1)] 
 %$L_{\mathcal{C}}( \mathfrak{c}(\Mod(D_n))) \asymp  \dfrac{1}{n^2}  $ and 
 $L_{\mathcal{C}}(\Mod(D_n)) \asymp  \dfrac{1}{n^2} $, and 
 \item[(2)] 
 $L_{\mathcal{C}}(\Mod(S_{1,n})) \asymp  \dfrac{1}{n^2}$. 
\end{enumerate}
\end{theorem}

Furthermore, we improve the upper bound for the minimal asymptotic translation length for $S_g$.
The \textit{hyperelliptic mapping class group} $\mathcal{H}(S_g)$ of $S_g$  
is the subgroup of $\Mod(S_g)$ consisting of elements with representative homeomorphisms
that commute with some fixed hyperelliptic involution.

\begin{theorem}\label{prop:upperbound}
For closed surfaces $S_g$ with $g \geq 3$,
$$L_{\mathcal{C}}(\mathcal{H}(S_g)) \leq \frac{1}{g^2-2g-1},$$
and as a direct consequence, we have
$$L_{\mathcal{C}}(\mathrm{Mod}(S_g)) \leq \frac{1}{g^2-2g-1}.$$
\end{theorem}
\noindent
We remark that for $g \geq 4$, this is a sharper upper bound than Gadre--Tsai's.

As another application, we determine the asymptotes of minimal asymptotic translation
lengths for handlebody groups and hyperelliptic handlebody groups.
Let $\mathbb{H}_g$ be the handlebody of genus $g$, that is, a 3-manifold bounded by a closed orientable surface $S_g$ of genus $g$.
The \textit{handlebody group} $\mathrm{Mod}({\mathbb{H}}_g)$ 
is the subgroup of  $\Mod(S_g)$ consisting of elements whose representative homeomorphisms of $S_g$
can be extended to homeomorphisms of $\mathbb{H}_g$.
Then the \textit{hyperelliptic handlebody group} is defined by
$$\mathcal{H}({\Bbb H}_g) = \mathrm{Mod}({\Bbb H}_g) \cap \mathcal{H}(S_g).$$

\begin{theorem}
\label{thm:hyphandle}
We have 
$$L_{\mathcal{C}} (\mathcal{H}({\Bbb H}_g)) \asymp \frac{1}{g^2}.$$
\end{theorem}

The following is an immediate corollary of the previous Theorem \ref{thm:hyphandle} and the lower bound by Gadre--Tsai.

\begin{corollary}
\label{cor:hyp}
We have 
$$L_{\mathcal{C}} (\mathcal{H}(S_g)) \asymp \frac{1}{g^2} \hspace{2mm}\mbox{and} \hspace{2mm} 
L_{\mathcal{C}} (\mathrm{Mod}({\Bbb H}_g) ) \asymp \frac{1}{g^2} .$$
\end{corollary}

\subsection*{Acknowledgement}
We thank Hyungryul Baik, Mladen Bestvina, Ki Hyoung Ko, Ken'ichi Ohshika, and Bal\'azs Strenner for helpful conversations. 
The first author was supported by 
Grant-in-Aid for
Scientific Research (C) (No. 15K04875), 
Japan Society for the Promotion of Science. 
The second author was supported by Basic Science Research Program
through the National Research Foundation of Korea(NRF) funded by the Ministry of Science, ICT and Future
Planning (NRF-2016R1C1B1006843).
We'd like to also thank the referees for valuable comments.

%%%%%%%%%%%%%%%%%%%%%%%%%%%%%%%%%%%%%%%%%%%%%%%%%%%%%%%%%%%%%
%																																												%
%				Proof of Main Theorem																																%
%																																										    	%
%%%%%%%%%%%%%%%%%%%%%%%%%%%%%%%%%%%%%%%%%%%%%%%%%%%%%%%%%%%%%

\vspace{1em}
\section{Proof of Theorem \ref{thm:main}}

In this section, we begin with the following simple observation.

\begin{lem} \label{lemma:distance}
Let $f \in \Mod(S)$ be a pseudo-Anosov mapping class and let $\alpha$ be any essential
simple closed curve in $S$. If $d_{\C}( \alpha, f^m(\alpha))=1$ for some $m \in \N$,
then 
$$\ell_{\C}(f) \leq \frac{1}{m}.$$
\end{lem}

\begin{proof}
By the triangle inequality, we have
{\setlength\arraycolsep{2pt}
\begin{eqnarray*}
\ell_{\C}(f^m) &=& \liminf_{j \ra \infty} \frac{d_{\C}(\alpha, f^{jm}(\alpha))}{j}\\
& \leq & \liminf_{j \ra \infty} \frac{ \sum_{i=1}^{j} d_{\C}(f^{(i-1)m}(\alpha), f^{im}(\alpha))}{j}\\
& = & \liminf_{j \ra \infty} \frac{j \cdot d_{\C}(\alpha, f^m(\alpha))}{j} = 1
\end{eqnarray*}}
Since $\ell_{\C}(f^m) = m \hspace{0.5mm} \ell_{\C}(f)$, this completes the proof.
\end{proof}

Now we prove our main theorem.

\begin{proof}[Proof of Theorem \ref{thm:main}]
Since the lower bound was established by Gadre--Tsai,
it is enough to show that there exists some constant $C$ such that
$$\ell_{\C}(\psi_n) \leq \frac{C}{|\chi(R_n)|^2}.$$

The proof consists of the following three steps. 
In step 1, we establish the structure of the $\Z$-cover $\widetilde{S}$ of $S$ 
as the union of $\Z$-copies of $S \setminus \{c\}$, where $[c]$ is the homology class
dual to the primitive cohomology $\xi_0$ fixed by $\psi$.
In step 2, using the decomposition of $\widetilde{S}$ in step 1, we will find
an integer $r$ so that $\psi_n^r(\overline{\alpha})$ and $\overline{\alpha}$ are disjoint 
in the quotient surface $R_n$, where 
$\overline{\alpha}$ is either a simple closed curve or a simple proper arc
and $\psi_n$ is a pseudo-Anosov monodromy on $R_n$.
In step 3, using Lemma \ref{lemma:distance}, we deduce that the asymptotic 
translation length of $\psi_n$ is less than or equal to $1/r$
and we show that we can choose $r$ to be quadratic in $n$. This finishes the proof.

\vspace{0.5em}
\textbf{Step 1.} (\textit{The decomposition of $\widetilde{S}$})
Let $[c]$ be a homology class in $H_1(S;\Z)$ which is dual to the primitive
cohomology $\xi_0 \in H^1(S;\Z)$.
Since $\xi_0$ is primitive, $[c]$ is also a primitive element.
If $S$ is a closed surface, one can find a representative $c$
that is an oriented simple closed curve
and if $S$ is a surface with punctures, $c$ can be chosen to be a simple proper arc
or union of disjoint simple proper arcs
(see, for instance, Proposition 6.2 in \cite{FarbMargalit12}).
Let $\widetilde{S}$ be the surface obtained by cutting $S$ along $c$ and
concatenating $\Z$-copies of $S \setminus \{c\}$ together
(see Figure \ref{figure:Z-cover} in the case of closed surfaces 
and see Figure \ref{fig_magic} in the case of punctured surfaces).
Then the natural projection map $p:\widetilde{S} \ra S$ is a covering map
corresponding to $\xi_0$ because
the kernel of the composition $\pi_1(S) \ra H_1(S;\Z) \xrightarrow{\xi_0} \Z$ 
of the Hurewicz map and $\xi_0$ is equal to $p_{\ast}(\pi(\widetilde{S}))$. 
Let $\Sigma_i$ be the copies of $S \setminus \{c\}$ on $\widetilde{S}$
such that the generator $h:\widetilde{S} \ra \widetilde{S}$ for the deck
transformation group is given by $h(\Sigma_i) = \Sigma_{i+1}$ for all $i$ (See Figure \ref{figure:Z-cover}).

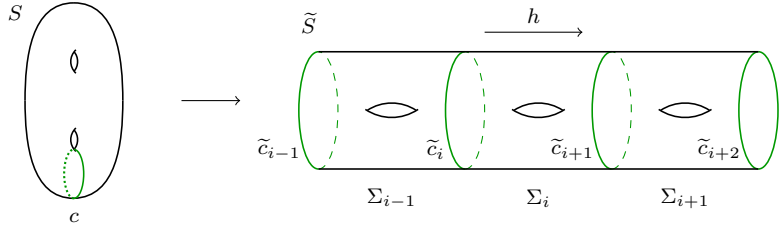
\begin{figure}[t]
\centering
\begin{tikzpicture}[scale=1.3]
		\draw[semithick] (0,1) to [out=90,in=180] (0.5,2);
		\draw[semithick] (0.5,2) to [out=0,in=90] (1,1);
		\draw[semithick] (0,1) to [out=-90,in=180] (0.5,0);
		\draw[semithick] (0.5,0) to [out=2,in=-90] (1,1);
		\draw[semithick] (0.5,1.3) .. controls (0.55,1.35) and (0.55,1.45) .. (0.5,1.5);
		\draw[semithick] (0.52,1.28) .. controls (0.45,1.35) and (0.45,1.45) .. (0.52,1.52);
		\draw[semithick] (0.52,0.48) .. controls (0.45,0.55) and (0.45,0.65) .. (0.52,0.72);
		\draw[semithick] (0.5,0.5) .. controls (0.55,0.55) and (0.55,0.65) .. (0.5,0.7);
		% curves
		\draw[thick,densely dotted,color=black!40!green] (0.5,0) arc (270:90:0.1 and 0.25);
		\draw[semithick,color=black!40!green] (0.5,0) arc (-90:90:0.1 and 0.25);
		\draw[->] (1.6,1) -- (2.2,1);
		% cover
		\draw[semithick,color=black!40!green] (3,0.3) arc (270:90:0.2 and 0.6);
		\draw[dashed,color=black!40!green] (3,0.3) arc (-90:90:0.2 and 0.6);
		\draw[semithick,color=black!40!green] (4.5,0.3) arc (270:90:0.2 and 0.6);
		\draw[dashed,color=black!40!green] (4.5,0.3) arc (-90:90:0.2 and 0.6);
		\draw[semithick,color=black!40!green] (6,0.3) arc (270:90:0.2 and 0.6);
		\draw[dashed,color=black!40!green] (6,0.3) arc (-90:90:0.2 and 0.6);
		\draw[semithick,color=black!40!green] (7.5,0.3) arc (270:90:0.2 and 0.6);
		\draw[semithick,color=black!40!green] (7.5,0.3) arc (-90:90:0.2 and 0.6);
		\draw[semithick] (3,0.3) -- (7.5,0.3);
		\draw[semithick] (3,1.5) -- (7.5,1.5);		
		\draw[semithick] (3.48,0.91) .. controls (3.65,0.8) and (3.85,0.8) .. (4.02,0.91);
		\draw[semithick] (3.5,0.9) .. controls (3.65,1) and (3.85,1) .. (4,0.9);		
		\draw[semithick] (5,0.9) .. controls (5.15,1) and (5.35,1) .. (5.5,0.9);
		\draw[semithick] (4.98,0.91) .. controls (5.15,0.8) and (5.35,0.8) .. (5.52,0.91);
		\draw[semithick] (6.5,0.9) .. controls (6.65,1) and (6.85,1) .. (7,0.9);
		\draw[semithick] (6.48,0.91) .. controls (6.65,0.8) and (6.85,0.8) .. (7.02,0.91);
		\draw[->] (4.7,1.7) -- (5.7,1.7) node [above,midway] {\footnotesize $h$};
		%\labels
		\draw node at (0.5, -0.2) {\footnotesize $c$};
		\draw node at (3.75, 0) {\footnotesize $\Sigma_{i-1}$};
		\draw node at (5.25, 0) {\footnotesize $\Sigma_{i}$};
		\draw node at (6.75, 0) {\footnotesize $\Sigma_{i+1}$};
		\draw node at (-0.1,1.9) {\footnotesize $S$};
		\draw node at (2.9,1.8) {\footnotesize $\widetilde{S}$};
		\draw node at (2.6,0.5) {\footnotesize $\widetilde{c}_{i-1}$};
		\draw node at (4.2,0.5) {\footnotesize $\widetilde{c}_{i}$};
		\draw node at (5.6,0.5) {\footnotesize $\widetilde{c}_{i+1}$};
		\draw node at (7.1,0.5) {\footnotesize $\widetilde{c}_{i+2}$};
\end{tikzpicture}
\caption{$\Z$-cover corresponding to $\xi_0$.}
\label{figure:Z-cover}
\end{figure}

\vspace{0.5em}
\textbf{Step 2.} (\textit{Finding a positive integer $r$ such that $\psi_n^r(\overline{\alpha})$ and $\overline{\alpha}$ are disjoint})
Choose a lift $\widetilde{\psi}$ and take a constant $k=k(\widetilde{\psi})$ 
such that
$$\widetilde{\psi}(\Sigma_0) \subset \Sigma_{-k} \cup \ldots \cup \Sigma_{k-1} \cup \Sigma_k.$$
(For instance, in Figure 3, $k=1$).
Note then we have $h^n \widetilde{\psi}(\Sigma_0) \subset \Sigma_{n-k} \cup \ldots \cup \Sigma_{n+k} $.

Suppose $n$ is large enough so that $n-k>1$. 
(More precise condition on $n$ will be determined later.) 
Then $h^n\widetilde{\psi}(\Sigma_0)$ and $\Sigma_0$ are disjoint and 
no orbit of $\Sigma_0$ under the cyclic group 
$\langle h^n\widetilde{\psi} \rangle$ intersect with 
$\Sigma_1 \cup \Sigma_2 \cup \ldots \cup \Sigma_{n-k-1}$. 
Let $\alpha$ be a simple closed curve or a simple proper arc contained in $\Sigma_0$
and let $\overline{\alpha}$ be
the $\langle h^n\widetilde{\psi} \rangle$-orbit of $\alpha$ in $\widetilde{S}$.
(One can choose $\alpha$ to be one of parallel copies of $c$, and $\alpha$ in Figure~\ref{figure:disjoint curves} is not the case.)
It follows that if a simple closed curve or a simple proper arc $\beta$ 
lies in $\Sigma_1 \cup \ldots \cup \Sigma_{n-k-1}$, i.e.,
disjoint from both $\alpha$ and $h^n\widetilde{\psi}(\alpha)$,
then $\overline{\beta}$ and $\overline{\alpha}$ are disjoint in 
$R_n = \widetilde{S}/\langle h^n\widetilde{\psi} \rangle$.
Since $h^{-1}$ induces a pseudo-Anosov map on $R_n$,
let us find a positive integer $r$ as large as possible such that
one of the representative of $\overline{h^{-r}(\alpha)}$ is contained in 
$\Sigma_1 \cup \ldots \cup \Sigma_{n-k-1}$ (see Figure \ref{figure:disjoint curves}).
Then this representative is disjoint from both $\alpha$ and $h^n\widetilde{\psi}(\alpha)$,
and because of the previous argument, 
$\overline{h^{-r}(\alpha)}$ and $\overline{\alpha}$ are disjoint in $R_n$.
By the fact that $h^{-1}$ descends to a psuedo-Anosov $\psi_n$ in $R_n$ together with
Lemma \ref{lemma:distance}, this allows us to obtain the upper bound for 
the asymptotic translation length of $\psi_n$.
To find such $r$, first note that since $\alpha$ is in $\Sigma_0$, we have
$\widetilde{\psi} (\alpha) \subset \Sigma_{-k} \cup \ldots \cup \Sigma_{k}$
and $\widetilde{\psi}^m  (\alpha) \subset \Sigma_{-mk} \cup \ldots \cup \Sigma_{mk}$
for any $m \in \N$.
Since the generator $h$ of the deck transformation group translates $\Sigma_i$'s,
after applying $h^{mk+1}$, we have $h^{mk+1} \widetilde{\psi}^m (\alpha) \subset \Sigma_1 \cup \ldots \cup \Sigma_{2mk+1}$. 
In order that $h^{mk+1} \widetilde{\psi}^m (\alpha)$ lies in
$\Sigma_1 \cup \ldots \cup \Sigma_{n-k-1}$, we require that $2km+1 \leq n-k-1$.
Let us choose the biggest such $m$, that is
$$m = \lfloor \frac{n-k-2}{2k} \rfloor$$
(note that the precise assumption on $n$ is $(n-k-2)/2k \geq 1$ because we want $m$ to be positive).
Since $\overline{\widetilde{\psi}(\alpha)} = \overline{h^{-n}(\alpha)}$ and hence
$\overline{h^{mk+1} \widetilde{\psi}^m (\alpha)} = \overline{h^{-(n-k)m+1}(\alpha)}$,
the desired integer for $\overline{h^{-r}(\alpha)}$ and $\overline{\alpha}$ being disjoint
is $r=(n-k)m-1$.

\begin{figure}[t] 
\centering
\begin{tikzpicture}
		% label
		\draw node at (-2.9,-0.8) {\footnotesize $\widetilde{c}_0$};
		\draw node at (-1.4,-0.8) {\footnotesize $\widetilde{c}_1$};
		\draw node at (2.6,-0.8) {\footnotesize $\widetilde{\psi}(\widetilde{c}_0)$};
		\draw node at (4.4,-0.8) {\footnotesize $\widetilde{\psi}(\widetilde{c}_1)$};
		\draw node at (3.5,-2.6) {\footnotesize $\Sigma_0$};
		\draw node at (5.1,-2.6) {\footnotesize $\Sigma_1$};
		\draw node at (2.1,-2.6) {\footnotesize $\Sigma_{-1}$};
		% shade
		\path [fill=gray!30] (2.78,-2.3) .. controls (2.5,-2.1) .. (2.2,-1.75);
		\path [fill=gray!30] (2.8,-1.1) .. controls (3,-1.4) .. (3.4,-1.65);
		\path [fill=gray!30] (2.2,-1.75) -- (2.78,-2.3) -- (2.8,-1.7);
		\path [fill=gray!30] (2.78,-2.3) -- (2.8,-1.7) -- (4.3,-1.7) -- (4.28,-2.3);
		\path [fill=gray!30] (2.2,-1.75) -- (2.8,-1.7) -- (2.2,-1.65);
		\path [fill=gray!30] (2.8,-1.1) -- (3.38,-1.64) -- (3.7,-1.65) -- (4.9,-1.75) -- (4.9,-1.65) -- (4.3,-1.1);
		\path [fill=white] (3.4,-1.75) -- (2.8,-1.7) -- (3.4,-1.7);
		\path [fill=white] (4.3,-1.7) -- (4.28,-2.3) -- (3.7,-1.75) -- (3.7,-1.65);
		\path [fill=white] (4.28,-2.3) .. controls (4,-2.1) .. (3.7,-1.75);
		\path [fill=white] (4.3,-1.1) .. controls (4.5,-1.4) .. (4.9,-1.65);
		% dotted arcs
		\draw[semithick,dashed,color=black!40!green] (2.2,-1.65) -- (2.8,-1.1);
		\draw[semithick,dashed,color=black!40!green] (3.4,-1.65) -- (2.2,-1.75);
		\draw[semithick,dashed,color=black!40!green] (3.4,-1.75) -- (2.8,-2.3);
		\draw[semithick,dashed,color=black!40!green] (3.7,-1.65) -- (4.3,-1.1);
		\draw[semithick,dashed,color=black!40!green] (4.9,-1.65) -- (3.7,-1.75);
		\draw[semithick,dashed,color=black!40!green] (4.9,-1.75) -- (4.3,-2.3);
		% surface
		\draw[semithick] (1.3,-2.3) arc (270:90:0.2 and 0.6);
		\draw[dashed] (1.3,-2.3) arc (-90:90:0.2 and 0.6);
		\draw[semithick] (2.8,-2.3) arc (270:90:0.2 and 0.6);
		\draw[dashed] (2.8,-2.3) arc (-90:90:0.2 and 0.6);
		\draw[semithick] (4.3,-2.3) arc (270:90:0.2 and 0.6);
		\draw[dashed] (4.3,-2.3) arc (-90:90:0.2 and 0.6);
		\draw[semithick] (5.8,-2.3) arc (270:90:0.2 and 0.6);
		\draw[semithick] (5.8,-2.3) arc (-90:90:0.2 and 0.6);
		\draw[semithick] (1.3,-2.3) -- (5.8,-2.3);
		\draw[semithick] (1.3,-1.1) -- (5.8,-1.1);
		\draw[semithick,fill=white] (1.78,-1.69) .. controls (1.95,-1.8) and (2.15,-1.8) .. (2.32,-1.69);
		\draw[semithick,fill=white] (1.8,-1.7) .. controls (1.95,-1.6) and (2.15,-1.6) .. (2.3,-1.7);	
		\draw[semithick,fill=white] (3.28,-1.69) .. controls (3.45,-1.8) and (3.65,-1.8) .. (3.82,-1.69);
		\draw[semithick,fill=white] (3.3,-1.7) .. controls (3.45,-1.6) and (3.65,-1.6) .. (3.8,-1.7);
		\draw[semithick,fill=white] (4.78,-1.69) .. controls (4.95,-1.8) and (5.15,-1.8) .. (5.32,-1.69);
		\draw[semithick,fill=white] (4.8,-1.7) .. controls (4.95,-1.6) and (5.15,-1.6) .. (5.3,-1.7);
		% boundary curves
		\draw[semithick,color=black!40!green] (2.78,-2.3) .. controls (2.5,-2.1) .. (2.2,-1.75);
		\draw[semithick,color=black!40!green] (3.4,-1.75) .. controls (2.8,-1.7) .. (2.2,-1.65);
		\draw[semithick,color=black!40!green] (2.8,-1.1) .. controls (3,-1.4) .. (3.4,-1.65);
		\draw[semithick,color=black!40!green] (4.28,-2.3) .. controls (4,-2.1) .. (3.7,-1.75);
		\draw[semithick,color=black!40!green] (4.9,-1.75) .. controls (4.3,-1.7) .. (3.7,-1.65);
		\draw[semithick,color=black!40!green] (4.3,-1.1) .. controls (4.5,-1.4) .. (4.9,-1.65);
		% arrow
		\draw[->] (0.45,-1.7) -- (0.85,-1.7);
		% surface -5.8
		\path [fill=gray!20] (-3.1,-2.3) -- (-1.5,-2.3) -- (-1.5,-1.1) -- (-3,-1.1);
		\draw[semithick] (-4.5,-2.3) arc (270:90:0.2 and 0.6);
		\draw[dashed] (-4.5,-2.3) arc (-90:90:0.2 and 0.6);
		\draw[semithick,color=black!40!green,fill=gray!20] (-3,-2.3) arc (270:90:0.2 and 0.6);
		\draw[dashed,color=black!40!green] (-3,-2.3) arc (-90:90:0.2 and 0.6);
		\draw[semithick,color=black!40!green,fill=white] (-1.5,-2.3) arc (270:90:0.2 and 0.6);
		\draw[dashed,color=black!40!green] (-1.5,-2.3) arc (-90:90:0.2 and 0.6);
		\draw[semithick] (0,-2.3) arc (270:90:0.2 and 0.6);
		\draw[semithick] (0,-2.3) arc (-90:90:0.2 and 0.6);
		\draw[semithick] (-4.5,-2.3) -- (0,-2.3);
		\draw[semithick] (-4.5,-1.1) -- (0,-1.1);
		\draw[semithick,fill=white] (-4.02,-1.69) .. controls (-3.85,-1.8) and (-3.65,-1.8) .. (-3.48,-1.69);
		\draw[semithick,fill=white] (-4,-1.7) .. controls (-3.85,-1.6) and (-3.65,-1.6) .. (-3.5,-1.7);	
		\draw[semithick,fill=white] (-2.52,-1.69) .. controls (-2.35,-1.8) and (-2.15,-1.8) .. (-1.98,-1.69);
		\draw[semithick,fill=white] (-2.5,-1.7) .. controls (-2.35,-1.6) and (-2.15,-1.6) .. (-2,-1.7);
		\draw[semithick,fill=white] (-1.02,-1.69) .. controls (-0.85,-1.8) and (-0.64,-1.8) .. (-0.492,-1.69);
		\draw[semithick,fill=white] (-1,-1.7) .. controls (-0.85,-1.6) and (-0.65,-1.6) .. (-0.5,-1.7);
		\draw node at (-3.6,-2.6) {\footnotesize $\Sigma_{-1}$};
		\draw node at (-2.2,-2.6) {\footnotesize $\Sigma_0$};
		\draw node at (-0.7,-2.6) {\footnotesize $\Sigma_1$};
		\draw node at (3.6,-1.33) {\footnotesize $\widetilde{\psi}(\Sigma_0)$};
\end{tikzpicture}\label{figure:choiceoflift}
\vspace{-1em}
\caption{A lift of $\psi$, where $\psi$ is given as in Figure \ref{figure:genus2}.
Curves $\widetilde{c}_0$ and $\widetilde{c}_1$, which are the lifts of $c$, determines
a fundamental region $\Sigma_0$.
Then the images $\widetilde{\psi}(\widetilde{c}_0)$ and $\widetilde{\psi}(\widetilde{c}_1)$,
lifts of $\psi(c)$,
bound the image $\widetilde{\psi}(\Sigma_0)$, which lies in $\Sigma_{-1} \cup \Sigma_0 \cup \Sigma_1$.
}

\vspace{1.5em}
\begin{tikzpicture}
		% shade
		\draw[fill=gray!20,gray!20] (1.5,0) rectangle +(6,1.2);
		% ellipses
		\draw[semithick] (0,0) arc (270:90:0.2 and 0.6);
		\draw[dashed] (0,0) arc (-90:90:0.2 and 0.6);
		\draw[semithick,fill= gray!20] (1.5,0) arc (270:90:0.2 and 0.6);
		\draw[dashed,fill= gray!20] (1.5,0) arc (-90:90:0.2 and 0.6);
		\draw[semithick,fill= gray!20] (3,0) arc (270:90:0.2 and 0.6);
		\draw[dashed,fill= gray!20] (3,0) arc (-90:90:0.2 and 0.6);
		\draw[semithick] (4.5,0) arc (270:90:0.2 and 0.6);
		\draw[dashed] (4.5,0) arc (-90:90:0.2 and 0.6);
		\draw[semithick] (6,0) arc (270:90:0.2 and 0.6);
		\draw[dashed] (6,0) arc (-90:90:0.2 and 0.6);
		\draw[semithick,fill=white] (7.5,0) arc (270:90:0.2 and 0.6);
		\draw[dashed,fill=white] (7.5,0) arc (-90:90:0.2 and 0.6);
		\draw[semithick] (9,0) arc (270:90:0.2 and 0.6);
		\draw[dashed] (9,0) arc (-90:90:0.2 and 0.6);
		\draw[semithick] (10.5,0) arc (270:90:0.2 and 0.6);
		\draw[dashed] (10.5,0) arc (-90:90:0.2 and 0.6);
		\draw[semithick] (12,0) arc (270:90:0.2 and 0.6);
		\draw[semithick] (12,0) arc (-90:90:0.2 and 0.6);
		% top and botton lines
		\draw[-,semithick] (0,0) -- (12,0);
		\draw[-,semithick] (0,1.2) -- (12,1.2);
		% genus
		\draw[semithick] (0.5,0.6) .. controls (0.65,0.7) and (0.85,0.7) .. (1,0.6);
		\draw[semithick] (0.48,0.61) .. controls (0.65,0.5) and (0.85,0.5) .. (1.02,0.61);
		\draw[semithick,fill=white] (1.98,0.61) .. controls (2.15,0.5) and (2.35,0.5) .. (2.52,0.61);
		\draw[semithick,fill=white] (2,0.6) .. controls (2.15,0.7) and (2.35,0.7) .. (2.5,0.6);		
		\draw[semithick,fill=white] (3.48,0.61) .. controls (3.65,0.5) and (3.85,0.5) .. (4.02,0.61);
		\draw[semithick,fill=white] (3.5,0.6) .. controls (3.65,0.7) and (3.85,0.7) .. (4,0.6);
		\draw[semithick,fill=white] (4.98,0.61) .. controls (5.15,0.5) and (5.35,0.5) .. (5.52,0.61);
		\draw[semithick,fill=white] (5,0.6) .. controls (5.15,0.7) and (5.35,0.7) .. (5.5,0.6);		
		\draw[semithick,fill=white] (6.48,0.61) .. controls (6.65,0.5) and (6.85,0.5) .. (7.02,0.61);
		\draw[semithick,fill=white] (6.5,0.6) .. controls (6.65,0.7) and (6.85,0.7) .. (7,0.6);
		\draw[semithick,fill=white] (7.98,0.61) .. controls (8.15,0.5) and (8.35,0.5) .. (8.52,0.61);
		\draw[semithick,fill=white] (8,0.6) .. controls (8.15,0.7) and (8.35,0.7) .. (8.5,0.6);		
		\draw[semithick] (9.5,0.6) .. controls (9.65,0.7) and (9.85,0.7) .. (10,0.6);
		\draw[semithick] (9.48,0.61) .. controls (9.65,0.5) and (9.85,0.5) .. (10.02,0.61);
		\draw[semithick] (11,0.6) .. controls (11.15,0.7) and (11.35,0.7) .. (11.5,0.6);
		\draw[semithick] (10.98,0.61) .. controls (11.15,0.5) and (11.35,0.5) .. (11.52,0.61);
		% curves
		\draw[thick,color=red] (0.75,0.6) ellipse (0.4 and 0.2);
		\draw[thick,color=red] (8.5,0.6) .. controls (9,0.2) .. (9.75,0);
		\draw[semithick,dashed,color=red] (9.75,0) .. controls (10.5,0.2) .. (11,0.6);
		\draw[semithick,dashed,color=red] (8.5,0.6) .. controls (9,1) .. (9.75,1.2);
		\draw[thick,color=red] (9.75,1.2) .. controls (10.5,1) .. (11,0.6);
		\draw[thick,color=blue] (3.8,1) .. controls  (1.1,1.2) and (1.1, -0.3) .. (3.8,0.85);
		\draw[dotted,thick,color=blue] (4,0.925) -- (6,0.925);
		\draw[thick,color=blue] (5.2,0.2) .. controls  (7.9,0) and (7.9, 1.5) .. (5.2,0.35);
		\draw[dotted,thick,color=blue] (3,0.275) -- (5,0.275);	
		% labels
		\draw node at (0.75, -0.3) {\footnotesize $\Sigma_0$};
		\draw node at (2.25, -0.3) {\footnotesize $\Sigma_1$};
		\draw node at (4.5,-0.3)    {\footnotesize $\cdots$};
		\draw node at (6.75, -0.3) {\footnotesize $\Sigma_{n-k-1}$};
		\draw node at (9.75, -0.3) {\footnotesize $\Sigma_n$};
		\draw node at (0.75, 1) {\footnotesize $\alpha$};
		\draw node at (9.85, 1.5) {\footnotesize $\widetilde{\psi}h^n(\alpha)$};
		\draw node at (4.5, 1.5) {\footnotesize $\overline{h^{-r}(\alpha)}$};
		\draw node at (0,1.5) {\footnotesize $\widetilde{S}$};
\end{tikzpicture}
\caption{The curves $\overline{h^{-r}(\alpha)}$ and $\overline{\alpha}$ are disjoint in $R_n = \widetilde{S}/\langle \widetilde{\psi}h^n \rangle$.}
\label{figure:disjoint curves}
\end{figure}

\vspace{0.5em}
\textbf{Step 3.} (\textit{Small asymptotic translation length $\ell_{\C}(\psi_n)$})
We first remark that arc and curve complex $\AC(S)$ and curve complex $\C(S)$
are 2-bilipschitz (see, for instance, \cite[Lemma 2.2]{MasurMinsky00} or \cite{HenselPrzytyckiWebb15}).
This implies that the asymptotic translation lengths $\ell_{\AC}(f)$ and $\ell_{\C}(f)$ of 
a pseudo-Anosov mapping class $f$ on the $1$-skeletons $\AC^1(S)$ and $\C^1(S)$,
respectively, have the same asymptotic behavior, that is,
$$\ell_{\AC}(f) \asymp \ell_{\C}(f).$$
So without loss of generality, we may assume that $\alpha$ is a simple closed curve and
compute the asymptotic translation length on the curve complex $\C(R_n)$.
In the case when $\alpha$ is a simple proper arc, we think of computing 
the asymptotic translation length of the arc and curve complex $\AC(R_n)$ and it gives us
the same asymptotic behavior on the curve complex.

By the previous step, $\psi_n^{(n-k)m-1}(\overline{\alpha})$ and $\overline{\alpha}$ are disjoint in $R_n$. Then by Lemma \ref{lemma:distance}, we have
$$\ell_{\C}(\psi_n) \leq \frac{1}{(n-k)m-1},$$
and by the fact that $\chi(R_n)$ is a linear function in $n$, we have 
$$\ell_{\C}(\psi_n) \leq \frac{C}{|\chi(R_n)|^2}$$
for some $C>0$.
This completes the proof.
\end{proof}

\vspace{1em}

\section{Backgrounds for Theorems~\ref{thm:sphere-torus}, \ref{prop:upperbound}, and \ref{thm:hyphandle}.}

This section includes some backgrounds and basic facts for the proofs of the rest of theorems. 
Consider a pseudo-Anosov mapping class $\psi \in \mathrm{Mod}(S)$. 
Let $\Psi: S \rightarrow S$ be any representative of $\psi$. 
The {\it mapping torus} $M_{\psi}$ is defined by 
$$M_{\psi}= S \times [0,1]/ \sim,$$
where $\sim$ identifies $(x,1)$ with $(\Psi(x),0)$ for each $x \in S$. 
Then the manifold $M_{\psi}$ fibering over the circle $S^1$ is  hyperbolic. 
%by Thurston's hyperbolization theorem. 
Suppose that  there is a primitive cohomology class $\xi_0 \in H^1(S; {\Bbb Z})$ fixed by $\psi$. 
This implies that  $b_1(M_{\psi}) \ge 2$. 
Then Theorem~\ref{thm:main} 
says that for $n$ sufficiently large, 
$R_n=  \widetilde{S}/\langle h^n\widetilde{\psi} \rangle$ is a fiber of $M_{\psi}$ 
with $\chi(R_n) \asymp n$ such that 
the pseudo-Anosov monodromy $\psi_n$ defined on the fiber $R_n$ satisfies 
$\ell_{\mathcal{C}}(\psi_n) \asymp 1/|\chi(R_n)|^2$.

\subsection{Fibered $3$-manifolds  from braids}
\label{subsection_fibered}

Let $B_n$ be the  the braid group with $n$ strands. 
In this paper braids are depicted vertically. 
We define the  product $\beta \beta'$ of $\beta, \beta' \in B_n$ in the usual way,
namely,  we stack  $\beta$ on  $\beta'$ and concatenate the bottom $i$th end point of 
$\beta$ with the top $i$th end point of $\beta'$ for each $i= 1, \cdots, n$. 
Then we obtain $n$ strands.  
The product $\beta \beta' $ is the resulting $n$-braid after rescaling.

We briefly review a relation between $B_n$ and $\mathrm{Mod}(D_n)$. 
To do this we assign an orientation for each $n$-braid  from the bottom endpoints to the top endpoints (see 
Figure~\ref{fig_nbraid}(2)). 
We take a natural basis $t_i  \in H_1(D_n; {\Bbb Z})$, 
where a representative of $t_i$ is a small oriented loop in $D_n$ centered at the $i$th puncture of $D_n$ for $i= 1, \cdots ,n$.   
Let $c_i$  be a simple proper arc in $D_n$ which connects the $i$th puncture of $D_n$ to the boundary $\partial D$ as in Figure~\ref{fig_nbraid}(1). 
Then there is an isomorphism 
$$\Gamma: B_n \rightarrow \mathrm{Mod}(D_n)$$ 
which sends the generator $\sigma_i$ of $B_n$ to the left-handed half twist $h_i$ (see Figure~\ref{fig_nbraid}(2)(3)). 
The orientation of braids as we described above induces the motion of $n$ punctures in the disk, 
which defines the above map $\Gamma$.

\begin{figure}
\centering
\includegraphics[width=3in]{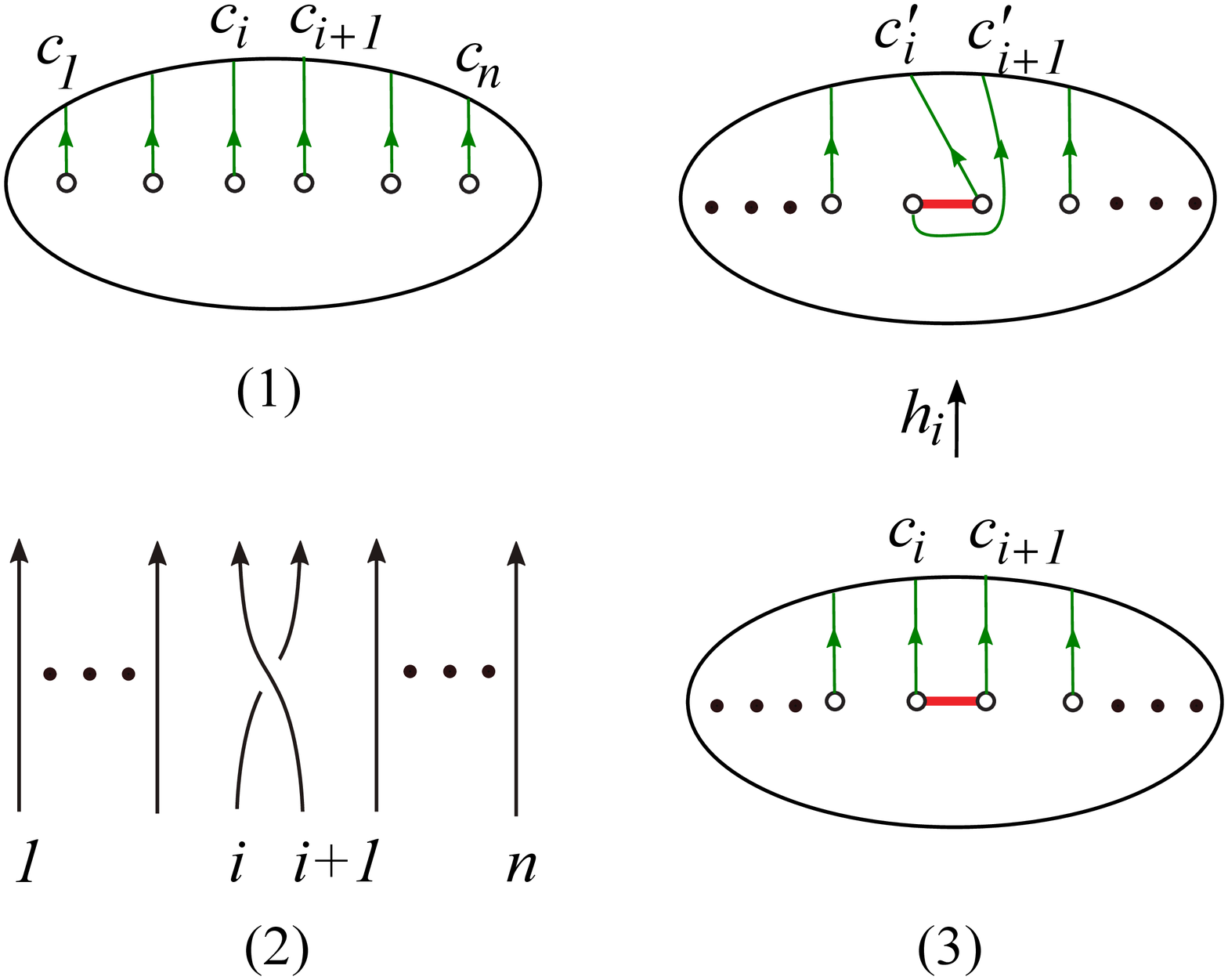}
\caption{(1) Arcs $c_i$  in the $n$-punctured disk $D_n$. 
(2) Generators $\sigma_i  $. 
(3) Half twist $h_i $. 
($c_i'= h_i(c_i)$ and $c_{i+1}'= h_i(c_{i+1})$.)} 
\label{fig_nbraid}
\end{figure}
We have a  natural homomorphism 
$$\mathfrak{c} :  \mathrm{Mod}(D_n) \rightarrow \mathrm{Mod}(S_{0,n+1})$$  
collapsing the boundary $\partial D$ of the disk to the $(n+1)$th puncture of $S_{0,n+1}$. 
By definition, $ \mathfrak{c}( \mathrm{Mod}(D_n))$ is isomorphic to the subgroup 
of $\mathrm{Mod}(S_{0,n+1})$ which fixes this puncture. 
We sometimes identify   $f \in  \mathrm{Mod}(D_n) $ with $\mathfrak{c}(f) \in  \mathrm{Mod}(S_{0,n+1})$. 
We simply denote by $\beta$, 
both mapping classes $\Gamma(\beta) \in \mathrm{Mod}(D_n)$ and 
$\mathfrak{c}(\Gamma(\beta)) \in \mathrm{Mod}(S_{0,n+1})$.

The closure $\mathrm{cl}(\beta)$ of   $\beta \in B_n$ is a knot or link in the $3$-sphere $S^3$. 
Let $\mathcal{A}$ be a braid axis of $\beta$ which is an unknot in $S^3$. 
Then $\mathrm{cl}(\beta)$ runs around the unknot $\mathcal{A}$ in a monotone manner. 
We set $\mathrm{br}(\beta)=  \mathrm{cl}(\beta) \cup \mathcal{A}$ which is a link in $S^3$ 
whose number of the  components is greater than or equal to $ 2$, 
and let  us set $M_{\beta} = S^3 \setminus \mathrm{br}(\beta)$. 
The $3$-manifold $M_{\beta}$ is homeomorphic to the interior of the mapping torus of the monodromy $\beta \in \mathrm{Mod}(D_n)$, 
and $b_1(M_{\beta}) \ge 2$. 
A spanning disk by the unknot $\mathcal{A}$ has $n$ punctures in $M_{\beta} $, and 
such a disk with punctures is a fiber of $M_{\beta}$ with monodromy $\beta$.

\begin{center}
\begin{figure}[t]
\includegraphics[width=5in]{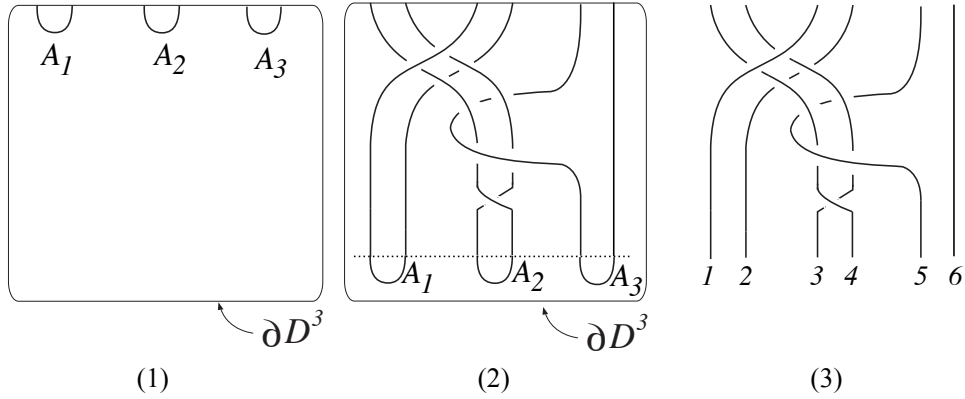}
\caption{(1) ${\bf A}$ in the case $n= 3$. (2) $^{w}{\bf A}$. (3)  $w \in SW_6 < SB_6$.} 
\label{fig_wicket}
\end{figure}
\end{center}

\begin{figure}[t]
\centering
\includegraphics[width=4.5in]{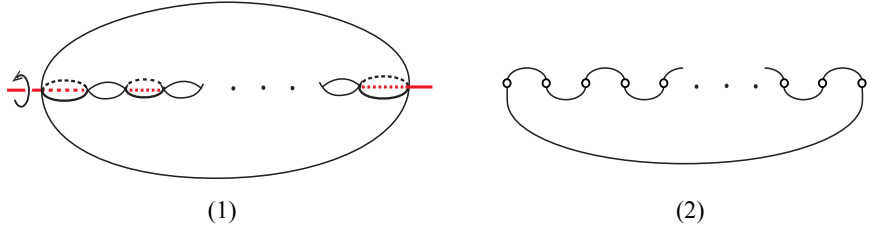}
\caption{(1) $\mathcal{I}: S_g \rightarrow S_g$. 
(2) Sphere $S_g/\mathcal{I}$ with $2g+2$ marked points. 
Small circles in the figure indicate marked points.} 
\label{fig_iota}
\end{figure} 

\vspace{-2em}

\subsection{Subgroups of mapping class groups}
\label{subsection_subgroup}

Let $SB_m$ be the {\it spherical} $m$-braid group. 
We now introduce the subgroup $SW_{2n}$ of $SB_{2n}$. 
Let $A_1, A_2, \cdots, A_n$ be  $n$ disjoint unknotted arcs properly embedded in the $3$-ball $D^3$ 
so that ${\bf A} = A_1 \cup \cdots \cup A_{n}$ is unlinked as in Figure~\ref{fig_wicket}(1). 
The boundary $\partial {\bf A}$ is the set  of $2n$ points in the $2$-sphere $\partial D^3$.

For $b \in SB_{2n}$, we stack $b$ on ${\bf A} $, 
and concatenate the bottom endpoints of $b$ with the endpoints of ${\bf A}$. 
As a result we obtain $n$ disjoint (knotted) arcs  $^{b}{\bf A}$ 
properly embedded in $D^3$ (see Figure~\ref{fig_wicket}(2)).
The {\it wicket group} $SW_{2n}$ is the subgroup of $SB_{2n}$ 
generated by braids $b$'s such that 
$^{b}{\bf A}$ is isotopic to ${\bf A}$ relative to $\partial {\bf A}$. 
It is easy to see that  the braid $w \in SB_6$ as shown in Figure~\ref{fig_wicket}(3) is an element of $SW_6$.  

There is a spherical version of the isomorphism $\Gamma: B_n \rightarrow \mathrm{Mod}(D_n)$, 
namely we have a surjective homomorphism 
$SB_m \rightarrow \mathrm{Mod}(S_{0,m})$ 
which sends the generator $\sigma_i$ of $SB_m$ to the left-handed half twist between the $i$th and $(i+1)$st punctures 
(cf. Figure~\ref{fig_nbraid}(2)(3)). 
We also denote this homomorphism by 
$$\Gamma: SB_m \rightarrow \mathrm{Mod}(S_{0,m}).$$ 
Its kernel  is isomorphic to ${\Bbb Z}/2{\Bbb Z}$ generated by a full twist $\Delta^2 \in SB_m$, 
where $\Delta$ is a half twist.  
When $m=2n$ the image $\Gamma(SW_{2n}) $ of $SW_{2n}$ under the map $\Gamma$ 
is a subgroup of $ \mathrm{Mod}(S_{0,2n})$ which is so-called   {\it Hilden group}, denoted by $SH_{2n}$, and 
$$SH_{2g+2} \simeq SW_{2g+2}/  \langle \Delta^2 \rangle$$ 
holds (see \cite{HiroseKin17}).

For the proof of Theorem~\ref{thm:hyphandle}, 
we recall a connection between the wicket group and the hyperellitic handlebody group. 
We first state a theorem by Birman and Hilden  
which relates $\mathcal{H}(S_g)$ to $\mathrm{Mod}(S_{0,2g+2})$. 
Each homeomorphism on $S_g$ which commutes with some fixed hyperelliptic involution $\mathcal{I}: S_g \rightarrow S_g$ 
(Figure~\ref{fig_iota}(1)) preserves the set of fixed points of $\mathcal{I}$ consisting of $2g+2$ points.
Such a homeomorphism induces a homeomorphism on a sphere $S_g/ \mathcal{I}$ 
which preserves these fixed points (Figure~\ref{fig_iota}(2)). 
Thus we have a map 
$$q: \mathcal{H}(S_g) \rightarrow \mathrm{Mod}(S_{0,2g+2})$$ 
by choosing a representative 
of each mapping class of $\mathcal{H}(S_g)$ 
which commutes with $\mathcal{I}$. 
It is shown in \cite{BirmanHilden71} that the map $q$ is well-defined and 
it is a surjective homomorphism whose kernel is generated by $\iota = [\mathcal{I}] \in \mathcal{H}(S_g)$. 
In particular we have 
$$\mathcal{H}(S_g) / \langle \iota \rangle \simeq \mathrm{Mod}(S_{0,2g+2}) 
\simeq SB_{2g+2}/ \langle \Delta^2 \rangle.$$

On the other hand, it is proved in \cite{HiroseKin17}  that 
there is a surjective homomorphism 
$$Q: \mathcal{H}({\Bbb H}_g) \rightarrow SH_{2g+2}$$ 
whose kernel is generated by $\iota $. 
The map $Q$ is given by the restriction 
$$q|_{ \mathcal{H}({\Bbb H}_g)}:  \mathcal{H}({\Bbb H}_g) \rightarrow SH_{2g+2} < \mathrm{Mod}(S_{0,2g+2}).$$ 
Putting all things together, we have 
$$\mathcal{H}({\Bbb H}_g) / \langle \iota \rangle \simeq SH_{2g+2} \simeq SW_{2g+2}/  \langle \Delta^2 \rangle.$$
Thus an element $f \in SH_{2g+2}$ can be described by a braid $v \in SW_{2g+2}$, i.e., $f= \Gamma(v)$. 
Moreover  
a lift $\widehat{f}$ of $f$ under the map $q|_{ \mathcal{H}({\Bbb H}_g)}= Q$ is an element of $\mathcal{H}({\Bbb H}_g)$. 
We simply denote by $v$, 
the element $\Gamma(v) $ in the Hilden group $SH_{2g+2}$.

The following lemma is used in the proofs of the rest of theorems (other than Theorem~\ref{thm:sphere-torus}(2)).

\begin{lem}
\label{lem_lift}
Let $f \in \mathrm{Mod}(S_{0,2g+2})$ for $g \ge 2$ and let $\widehat{f} \in \mathcal{H}(S_g)$ be a lift of $f$ 
under the map $q: \mathcal{H}(S_g) \rightarrow \mathrm{Mod}(S_{0,2g+2})$. 
We take any $\alpha \in \mathcal{AC}^0(S_{0, 2g+2})$, i.e.,
$\alpha$ is a homotopy class of an arc or simple closed curve in $S_{0, 2g+2}$. 
Suppose that  $d_{\mathcal{AC}}(\alpha, f^m(\alpha))= 1$ for some $m \in {\Bbb N}$, 
where $d_{\mathcal{AC}}$ is the path metric on $ \mathcal{AC}(S_{0, 2g+2})$. 
Then 
$$\ell_{\mathcal{C}}(\widehat{f}) \le \frac{1}{m}.$$
\end{lem}

It is well-known and not hard to see that 
if $f \in \mathrm{Mod}(S_{0,2g+2})$ is pseudo-Anosov, then $\widehat{f} \in \mathcal{H}(S_g)$ is also pseudo-Anosov.

\begin{proof}[Proof of Lemma~\ref{lem_lift}]
By abuse of the notation, a representative of $\alpha \in \mathcal{AC}^0(S_{0, 2g+2})$ is denoted by the same $\alpha$. 
Let $\widehat{\alpha} \subset S_g$ be the preimage $q^{-1}(\alpha)$ 
of a simple arc or simple closed curve $\alpha$ in $S_{0,2g+2}$ under the map $q$. 
If $\alpha$ is a simple arc, then $\widehat{\alpha}$ is a non-separating simple closed curve in $S_g$ 
which means $\widehat{\alpha}$ is essential. 
Hence $\widehat{\alpha} \in \mathcal{C}^0(S_g)$. 
The assumption implies that 
$d_{\mathcal{C}} (\widehat{\alpha}, (\widehat{f})^m(\widehat{\alpha}))= 1$. 
The claim follows from Lemma~\ref{lemma:distance}. 

If $\alpha$ is a simple closed curve, 
then $\alpha$ cuts $S_{0,2g+2}$ into two components $S_{(1)}$ and $S_{(2)}$ 
which are disks with punctures $n_1 \ge 2$ and $n_2 \ge 2$ respectively. 
Since $n_1+ n_2= 2g+2$,  both $n_1$ and $n_2$ have the same parity. 

We first consider the case where  $n_1$ and $n_2$ are odd. 
Then $n_1,n_2 \ge 3$. 
Observe that  $\widehat{\alpha}$ is a single simple closed curve. 
Since $\widehat{\alpha}$ cuts $S_g$ into the essential surfaces 
$q^{-1}(S_{(1)})$ and $q^{-1}(S_{(2)})$  with positive genera, 
$\widehat{\alpha}$ is a separating and essential simple closed curve. 
We have $d_{\mathcal{C}} (\widehat{\alpha}, (\widehat{f})^m(\widehat{\alpha}))= 1$ by the assumption. 
Thus $\ell_{\mathcal{C}}(\widehat{f}) \le \frac{1}{m}$ holds. 

Let us consider the remaining case where both $n_1$ and $n_2$ are even with $n_1, n_2 \ge 2$. 
Observe that $\widehat{\alpha}$ has two components $\widehat{\alpha}_{(1)}$ and $\widehat{\alpha}_{(2)}$ 
which are non-separating simple closed curves. 
Hence $\widehat{\alpha}_{(i)} \in \mathcal{C}^0(S_g)$ for $i= 1,2$. 
We have $d_{\mathcal{C}} (\widehat{\alpha}_{(i)}, (\widehat{f})^m(\widehat{\alpha}_{(i)}))= 1$ by the assumption, 
and hence $\ell_{\mathcal{C}}(\widehat{f}) \le \frac{1}{m}$ holds. 
We complete the proof. 
\end{proof}

\section{Proof of Theorem~\ref{thm:sphere-torus}}

\begin{figure}[tp]
\centering
\includegraphics[width=3.5in]{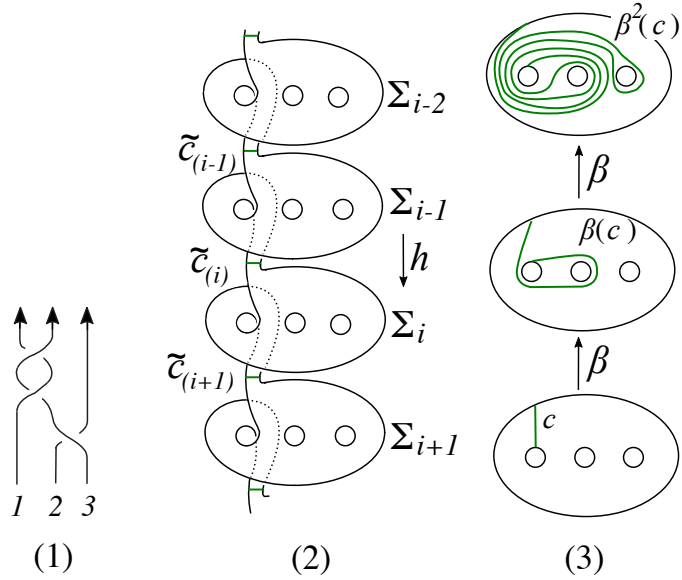}
\caption{
(1) $\beta = \sigma_1^{-2} \sigma_2 \in B_3$. 
(2) ${\Bbb Z}$-cover $\widetilde{S} $ over $S= D_3$ corresponding to the dual to $c= c_1$. 
(3) $c$, $\beta(c)$ and $ \beta^2(c)$.} 
\label{fig_magic}
\end{figure}

\begin{figure}[tp]
\centering
\includegraphics[width=3.8in]{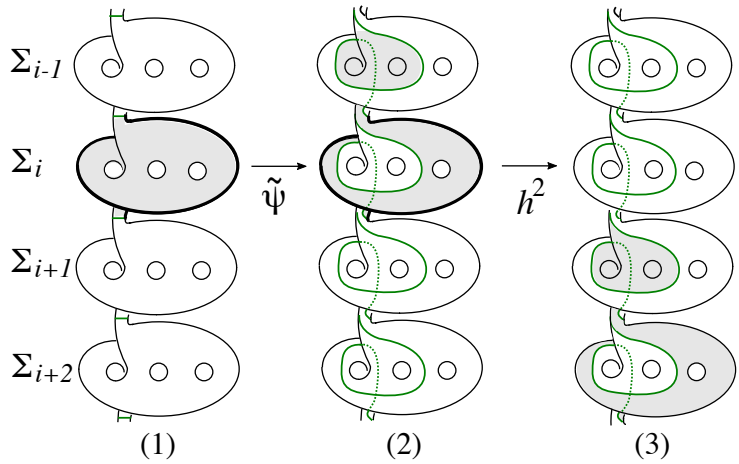}
\caption{
Illustration of $h^2 \widetilde{\psi}: \widetilde{S} \rightarrow  \widetilde{S}$. 
Shaded regions in (1)(2) and (3) are $\Sigma_i$, $\widetilde{\psi}(\Sigma_i)$ and 
$h^2 \widetilde{\psi}(\Sigma_i)$ respectively.} 
\label{fig_3braid_cover}
\end{figure}

This section is devoted to the proof of Theorem~\ref{thm:sphere-torus}. 
In the proof of Theorem~\ref{thm:sphere-torus}(2), 
we reprove the previous result 
$ \mathrm{Mod}(S_{1,2n}) \asymp \frac{1}{n^2}$ by Gadre--Tsai.

\begin{proof}[Proof of Theorem \ref{thm:sphere-torus}(1)] 
We separate the proof into two cases, 
depending on the parity of the number of punctures of $D_n$. 
We first deal with the case where $n$ is even. 
We consider the pseudo-Anosov braid $\beta = \sigma_1^{-2} \sigma_2 \in B_3$ 
(Figure~\ref{fig_magic}(1)) and 
the fibered hyperbolic $3$-manifold $M_{\beta}$. 
We take a fiber $S= D_3$ with monodromy $\psi= \beta$ of $M_{\beta}$. 
Let $\xi_0 \in H^1(S; {\Bbb Z})$ be the primitive cohomology class which is dual to the 
homology class of the proper arc $c= c_1$ in $S$ (see Figure~\ref{fig_nbraid}(1) for $c_1$). 

From Figure~\ref{fig_magic}(1), 
one sees that 
the induced h $\psi_*= \beta_*: H_1(D_3; {\Bbb Z}) \rightarrow H_1(D_3; {\Bbb Z})$ 
maps  $t_1, t_2, t_3 \in H_1(D_3; {\Bbb Z})$ to $t_1$, $t_3$, $t_2$ respectively, 
where the set of $t_i$'s is a natural basis of $ H_1(D_n; {\Bbb Z})$ (see Section~\ref{subsection_fibered}). 
This tells us that $\xi_0$ is fixed by  $\psi$. 
Figure~\ref{fig_magic}(2) illustrates the ${\Bbb Z}$-cover $\widetilde{S}$ corresponding to $\xi_0$. 
We consider the {\it canonical} lift $\widetilde{\psi}: \widetilde{S} \rightarrow \widetilde{S} $ of $\psi$ 
which means that $\widetilde{\psi}$ fixes the preimage $p^{-1}(\partial D)$ of the (outer) boundary of the $3$-punctured disk pointwise. 
(In Figure~\ref{fig_3braid_cover}(1)(2), the set $p^{-1}(\partial D) \cap \Sigma_i$ is thickened.) 
We set $\widetilde{c}_{(i)} = \Sigma_{i-1} \cap \Sigma_i$ which is a connected component of the preimage $p^{-1}(c)$ of $c$
(see Figure~\ref{fig_magic}(2)). 
In other words, 
$\widetilde{c}_{(i)} $ and $\widetilde{c}_{(i+1)}$ bound the copy $\Sigma_i$. 
To see the image $\widetilde{\psi} (\Sigma_i)$ of $\Sigma_i$ under $\widetilde{\psi}$, 
we consider $\widetilde{\psi} (\widetilde{c}_{(i)}) $ and $\widetilde{\psi} (\widetilde{c}_{(i+1)})$ 
which are determined by the proper arc $\psi(c)= \beta(c)$ (see Figure~\ref{fig_magic}(3)). 
Observe (from Figure~\ref{fig_3braid_cover}(1) and (2)) that 
$$\widetilde{\psi}(\Sigma_i)  \subset \Sigma_{i-1} \cup \Sigma_i \hspace{2mm} \mbox{and} \hspace{2mm}  
\widetilde{\psi}^{-1}(\Sigma_i)  \subset \Sigma_{i} \cup \Sigma_{i+1}. $$
Hence for each $n \ge 0$ 
$$h^n \widetilde{\psi} (\Sigma_i) 
 \subset \Sigma_{i-1 +n} \cup \Sigma_{i + n} \hspace{1mm} \mbox{and} \hspace{1mm} 
 (h^n \widetilde{\psi})^{-1}(\Sigma_i) = h^{-n}\widetilde{\psi}^{-1} (\Sigma_i) 
 \subset  \Sigma_{i -n} \cup \Sigma_{i- n+1}.$$
For $\ell >0$, we have 
\begin{eqnarray*}
(h^n \widetilde{\psi})^{\ell} (\Sigma_i) 
 &\subset& \Sigma_{i-\ell+ \ell n} \cup \cdots \cup \Sigma_{i-1 +\ell n } \cup \Sigma_{i+ \ell n}, 
 \\
 (h^n \widetilde{\psi})^{-\ell} (\Sigma_i) 
  &\subset& \Sigma_{i - \ell n} \cup \Sigma_{i - \ell n +1} \cup \cdots \cup \Sigma_{i- \ell n + \ell}.  
\end{eqnarray*}
Notice that if we fix $n \ge 2$, then 
$(h^n \widetilde{\psi})^{\pm \ell} (\Sigma_i)  \cap \Sigma_i = \emptyset$ for each $\ell >0$, and hence 
$R_n = \widetilde{S}/\langle h^n\widetilde{\psi} \rangle$ is a surface. 
In fact $R_n$ is a disk with $2n$ punctures, and hence we can think of $R_n$ as a sphere with $2n+1$ punctures
(see Figures~\ref{fig_3braid_cover} and \ref{fig_rn}). 
Note that one of the punctures of $R_n$, say $p_{\infty}$, comes from the preimage of the boundary of the disk 
under the projection $p: \widetilde{S} \rightarrow S = D_3$. 
By Theorem~\ref{thm:McMullen}, we know $h^{-1}$ descends to the monodromy $\psi_n$, and 
we see that $\psi_n$ maps $p_{\infty}$ to itself. 
Thus $\psi_n \in \mathrm{Mod}(D_{2n})$. 
By Theorem~\ref{thm:main}, we have 
$\ell_{\mathcal{C}}(\psi_n) \le C/n^2$ for some constant $C$, 
and hence $L_{\mathcal{C}}(\Mod(D_{2n}))  \le C/n^2$. 

\begin{figure}[tp]
\centering
\includegraphics[width=3.5in]{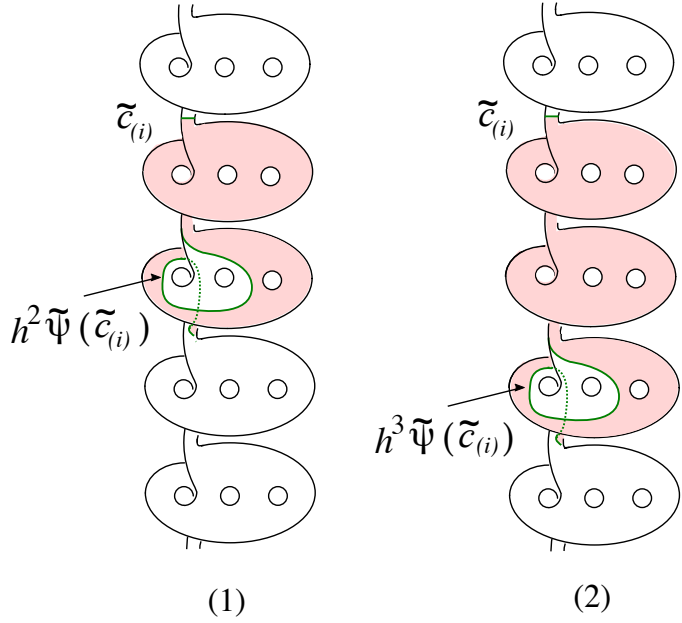}
\caption{(1) Shaded region  descends to 
$R_2  \simeq S_{0,5}$. See also Figure~\ref{fig_3braid_cover}. 
(2) Shaded region  descends to $R_3 \simeq S_{0,7}$. 
(Note that $[\widetilde{c}_{(i)}]= [h^n \widetilde{\psi}(\widetilde{c}_{(i)})]$ in $R_n$.)} 
\label{fig_rn}
\end{figure}

For the case where the number of the punctures of $D_n$ is odd, 
we turn to the pseudo-Anosov braid $\phi= \beta^2 \in B_3$. 
The hyperbolic fibered $3$-manifold $M_{\phi}$ has a fiber $S= D_3$ with monodromy $\phi$. 
The dual to $c= c_1$ is the primitive cohomology class fixed by  $\phi$. 
Consider the ${\Bbb Z}$-cover $\widetilde{S}$ corresponding to this cohomology class. 
Let $\widetilde{\phi}= (\widetilde{\psi})^2: \widetilde{S} \rightarrow \widetilde{S}$ 
be the canonical lift of $\phi$ as before. 
By using the proper arc $\phi(c)= \beta^2(c)$ (see Figure~\ref{fig_magic}(3)), 
we see where each copy $\Sigma_i $ maps on $\widetilde{S}$ under $\widetilde{\phi}$. 
We use  the same argument as above replacing $\widetilde{\psi}$ with $\widetilde{\phi}= (\widetilde{\psi})^2$, 
and construct a surface $\widetilde{S}/\langle h^n \widetilde{\phi} \rangle$ concretely. 
Then we find that this surface is a sphere with $2n+2$ punctures which is a fiber of $M_{\phi}$ for $n$ large. 
Also we see that $\phi_n$ fixes one of the  punctures of the fiber (which comes form the preimage of the boundary of the disk). 
Thus $\phi_n \in \mathrm{Mod}(D_{2n+1})$. 
By Theorem~\ref{thm:main}, we have 
$\ell_{\mathcal{C}}(\phi_n) \le C'/n^2$ for some constant $C'>0$. 
This tells us that  $L_{\mathcal{C}}(\Mod(D_{2n+1})) \le C'/n^2$. 
This completes the proof.
\begin{figure}[tp]
\labellist %\small\hair 2pt
\pinlabel $\ell$ at 128 38
\pinlabel $m$ at 108 90
\endlabellist
\centering
\includegraphics[width=3in]{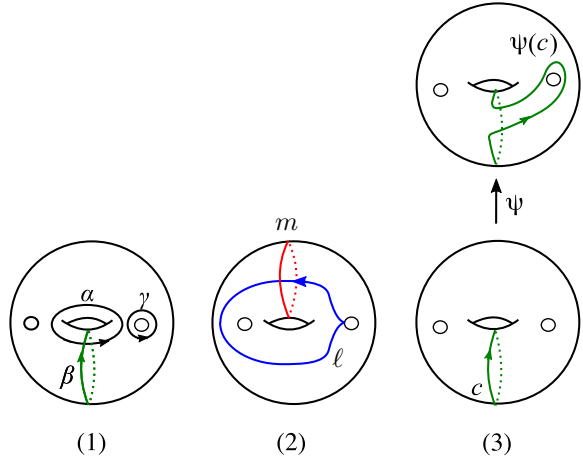}
\caption{Two small circles indicate punctures of $S_{1,2}$. 
(1) A basis $\alpha, \beta, \gamma \in H_1(S_{1,2}; {\Bbb Z})$. 
(2)  $m, \ell$  in $S_{1,2}$. 
(3) Image of $c$ under $\psi= T_m^{-1}  f_{\ell}$.} 
\label{fig_whitehead}
\end{figure}
\end{proof}

\begin{proof}[Proof of Theorem~\ref{thm:sphere-torus}(2)] 
We first consider the case where the number of punctures is odd. 
Let $L_W$ be the Whitehead link  in $S^3$. 
The complement $S^3 \setminus L_W$ is a fibered hyperbolic $3$-manifold with a fiber $S_{1,2}$. 
Consider its pseudo-Anosov monodromy $\psi$ defined on the fiber $S_{1,2}$
(see \cite[Appendix~B]{KinRolfsen16} for more details),
and we use a basis $\alpha, \beta, \gamma \in H_1(S_{1,2}; {\Bbb Z})$ as in Figure~\ref{fig_whitehead}(1). 
Let $m$ be a simple closed curve in $S_{1,2}$, 
and let $\ell$ be an oriented loop based at one of the punctures of  $S_{1,2}$ as in Figure~\ref{fig_whitehead}(2). 
Let $c$ be a representative of the generator $\beta \in H_1(S_{1,2}; {\Bbb Z})$ as in Figure~\ref{fig_whitehead}(3). 
We set  $\psi = T_m^{-1}  f_{\ell} \in \mathrm{Mod}(S_{1,2})$ 
where $f_{\ell}$ is the mapping class which represents the  point-pushing map along $\ell$
(see Figure~\ref{fig_whitehead}(3)). 
Then $\psi$ is the monodromy of a fibration on $S^3 \setminus L_W$, i.e.,
$M_{\psi}$ is homeomorphic to $S^3 \setminus L_W$. 
In particular $\psi$ is pseudo-Ansosov since $S^3 \setminus L_W$ is hyperbolic. 
Observe that the induced map $\psi_*: H_1(S_{1,2}; {\Bbb Z}) \rightarrow  H_1(S_{1,2}; {\Bbb Z}) $ 
sends $a, \beta$ and $\gamma$ to 
$ \alpha - \beta - \gamma$, $ \beta + \gamma$ and $\gamma$ respectively. 
Then the cohomology class $\xi_0 \in H^1(S_{1,2}; {\Bbb Z})$ which is dual to $c $, is primitive and fixed by $\psi$.  
We consider the ${\Bbb Z}$-cover $\widetilde{S}$ over $S= S_{1,2}$ corresponding to $\xi_0$, 
and we take a lift $\widetilde{\psi}: \widetilde{S} \rightarrow \widetilde{S}$ such that 
$\widetilde{\psi}(\Sigma_i) \subset \Sigma_{i-1} \cup \Sigma_i $ 
(see Figures~\ref{fig_torus_cover} and \ref{fig_whitehead}(3)). 
By the same argument as in the proof of Theorem~\ref{thm:sphere-torus}(1), we verify that 
$R_n$ is a torus with $2n+1$ punctures if $n \ge 2$. 
By Theorem~\ref{thm:main}, we conclude that   $L_{\mathcal{C}}(\Mod(S_{1,2n+1})) \le C/n^2$ for some constant $C>0$.

\begin{figure}[t]
\centering
\includegraphics[width=2.5in]{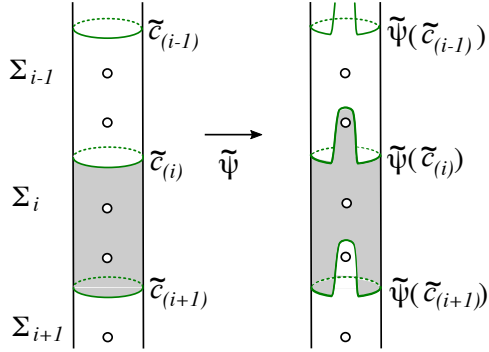}
\caption{A lift $\widetilde{\psi}: \widetilde{S} \rightarrow  \widetilde{S}$ of $\psi$ 
with $\widetilde{\psi}(\Sigma_i) \subset \Sigma_{i-1} \cup \Sigma_i $. 
The regions of $\Sigma_i$ and  $\widetilde{\psi}(\Sigma_i)$ are shaded.} 
\label{fig_torus_cover}
\end{figure}

\begin{figure}[t!]
\centering
\includegraphics[width=2.5in]{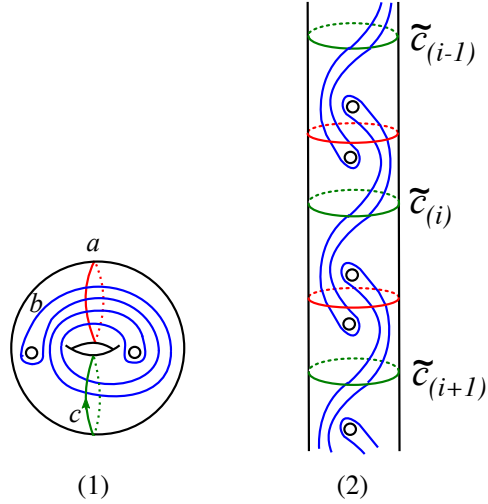}
\caption{(1) Simple closed curves $a,b$ in $S_{1,2}$. 
 (2) ${\Bbb Z}$-cover $\widetilde{S}$ over $S= S_{1,2}$ corresponding to the dual of $c$.} 
\label{fig_tsai}
\end{figure}

We turn to the case where the number of punctures is even. 
Let $a$ and $b$ be simple closed curves in $S_{1,2}$ as in Figure~\ref{fig_tsai}(1), and 
let $c$ be as before, i.e., $\beta= [c]$. 
Consider $\psi = T_b^{-1} T_a \in \mathrm{Mod}(S_{1,2})$ which is pseudo-Anosov by Penner's construction. 
The induced map $\psi_*$ maps a basis $a,\beta$ and $\gamma$ of $H_1(S_{1,2}; {\Bbb Z})$ to $\alpha+ \beta+ \gamma$, $\beta$, and $\gamma$, respectively. 
Thus $\psi$ fixes a primitive cohomology class $\xi_0 \in H^1(S_{1,2}; {\Bbb Z})$ which is dual to $c$. 
Consider  the ${\Bbb Z}$-cover $\widetilde{S}$ over $S$ corresponding to $\xi_0$ (Figure~\ref{fig_tsai}(2)) 
and pick a lift of $\widetilde{\psi}: \widetilde{S} \rightarrow \widetilde{S}$ of $\psi$. 
We can apply Theorem~\ref{thm:main} for 
the fiber $(S_{1,2}, \psi)$ of the mapping torus $M_{\psi}$  
together with $\xi_0 \in H^1(S_{1,2}; {\Bbb Z})$ fixed by $\psi$.  
Theorem~\ref{thm:McMullen} says that for all $n$ sufficiently large,  $R_n$ is a fiber of  $M_{\psi}$. 
In this case  $R_n$ is a torus with $2n+ n_0$ punctures, 
where $n_0 $ is an even number which  depends on the choice of the lift $\widetilde{\psi}$. 
By Theorem~\ref{thm:main} we conclude that  $L_{\mathcal{C}}(\Mod(S_{1,2n})) < C'/n^2$ for some constant $C'>0$. 
This completes the proof.  
\end{proof}

\section{Proof of Theorem~\ref{prop:upperbound}}
This section includes the proof of Theorem~\ref{prop:upperbound}.

\begin{figure}[t]
\centering
\includegraphics[width=3in]{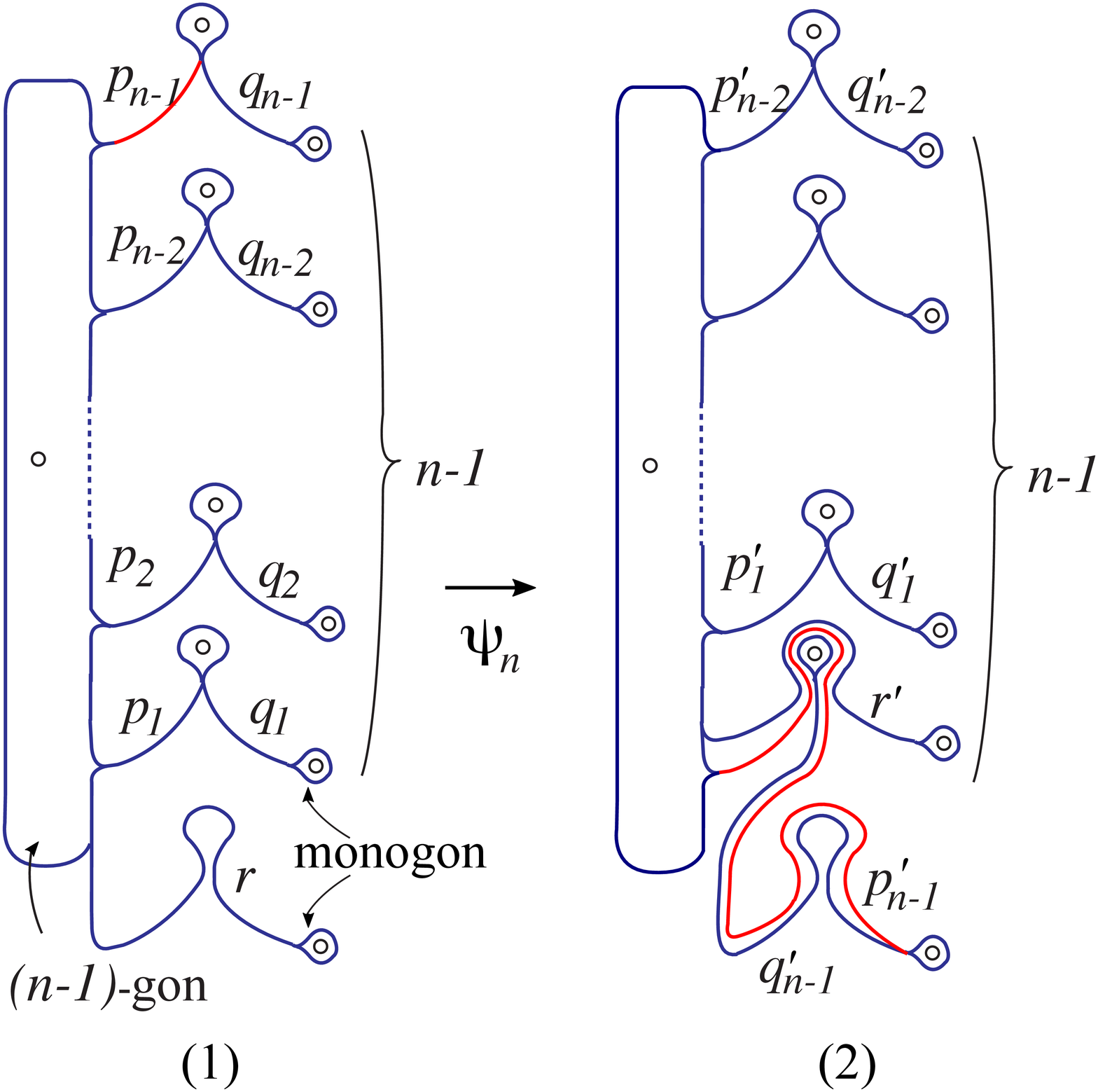}
\caption{Small circles indicate punctures of $S_{0,2n+1}$. 
(1) Train track $\tau_n$. (2) $\psi_n(\tau_n)$, where $e'= \psi_n(e)$. 
(The puncture $p_{\infty}$ is not drawn here.)} 
\label{fig_ttrackbraid}
\end{figure}

In the proof of Theorem~\ref{thm:sphere-torus}(1), 
we used the hyperbolic fibered $3$-manifold 
$M_{\beta}= M_{\sigma_1^{-2} \sigma_2}$, the so-called  {\it magic manifold} 
and its double cover $M_{\beta^2}$, 
depending on the parity of the number of punctures in the disk. 
Here we only use $M_{\beta}$ 
and a sequence $(R_n, \psi_n)$ of the fibers $R_n= D_{2n}$ of $M_{\beta}$ 
 with the monodromy $\psi_n$ for $n \ge 2$ 
as in the proof of Theorem~\ref{thm:sphere-torus}(1). 
{\it Train tracks} play an important role in the proof. 
Terminology related to train tracks can be found in \cite{BestvinaHandel95} or \cite{FarbMargalit12}
for example.

We think of $R_n$ as a sphere with $2n+1$ punctures. 
An invariant train track $\tau_n$ and a train track representative 
$\mathfrak{p}_n: \tau_n \rightarrow \tau_n$ 
of $\psi_n: S_{0,2n+1} \rightarrow S_{0,2n+1}$ 
are studied in \cite[Example~4.6]{Kin15}. 
Figure~\ref{fig_ttrackbraid} shows the train track $\tau_n \subset S_{0,2n+1}$ 
and its image $\psi_n(\tau_n)$. 
Each of the monogon components of $S_{0,2n+1} \setminus \tau_n$ 
(bounded by loop edges of $\tau_n$) contains a puncture of $S_{0,2n+1}$, 
the $(n-1)$-gon of $S_{0,2n+1} \setminus \tau_n$ contains another puncture, 
and the other connected component of $S_{0,2n+1} \setminus \tau_n$ contains the other puncture 
$p_{\infty}$ in the proof of Theorem~\ref{thm:sphere-torus}(1). 
Recall that 
$\psi_n$ maps $p_{\infty}$ to itself. 
Figure~\ref{fig_digraph} gives the directed graph $\Gamma_n$ of 
 $\mathfrak{p}_{n}: \tau_n \rightarrow \tau_n$ for $n \ge 3$. 
 The set of vertices of $\Gamma_n$ equals the set of  non-loop edges 
 $r$, $p_1$, $q_1$, $\cdots, p_{n-1}$, $q_{n-1}$  of $\tau_n$ 
 as shown in Figure~\ref{fig_ttrackbraid}.
 The edges of $\Gamma_n$ tell the locations of  
 $\mathfrak{p}_n(e)$, $\mathfrak{p}^2_n(e)$, $\mathfrak{p}^3_n(e),\cdots $ 
 in $S_{0,2n+1}$ 
 for each non-loop edge $e$ of $\tau_n$. 
More precisely, $j$ edges of $\Gamma_n$ running from the vertex $e$ to the vertex $e'$  mean that 
 $\mathfrak{p}_n(e)$ passes through the edge $e'$ of $\tau_n$ $j$ times. 
 One can construct  $\Gamma_n$ 
 viewing $\psi_n(\tau_n)$ and $\tau_n$. 
 The ``vertical" consecutive edges of $\Gamma_n$ in Figure~\ref{fig_digraph} 
 reveal the dynamics of $\psi_n: S_{0, 2n+1} \rightarrow S_{0,2n+1}$  
 which is just like a translation on a ``big" subsurface of $S_{0,2n+1}$. 

We first prove the following.

\begin{prop}
\label{prop_bound-sphere}
For $n \ge 4$, we have 
$$L_{\mathcal{C}}(\mathrm{Mod}(D_{2n-1})) \le \frac{1}{n^2-4n+2}\hspace{2mm}  \ \mbox{and}\hspace{2mm} 
L_{\mathcal{C}}( \mathrm{Mod}(D_{2n})) \le \frac{1}{n^2-4n+2}.$$
\end{prop}

\begin{figure}[t]
\centering
\includegraphics[width=1.5in]{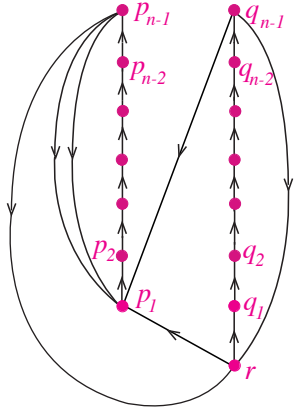}
\caption{Directed graph $\Gamma_n$.} 
\label{fig_digraph}
\end{figure}

\begin{figure}[t]
\centering
\includegraphics[width=3.5in]{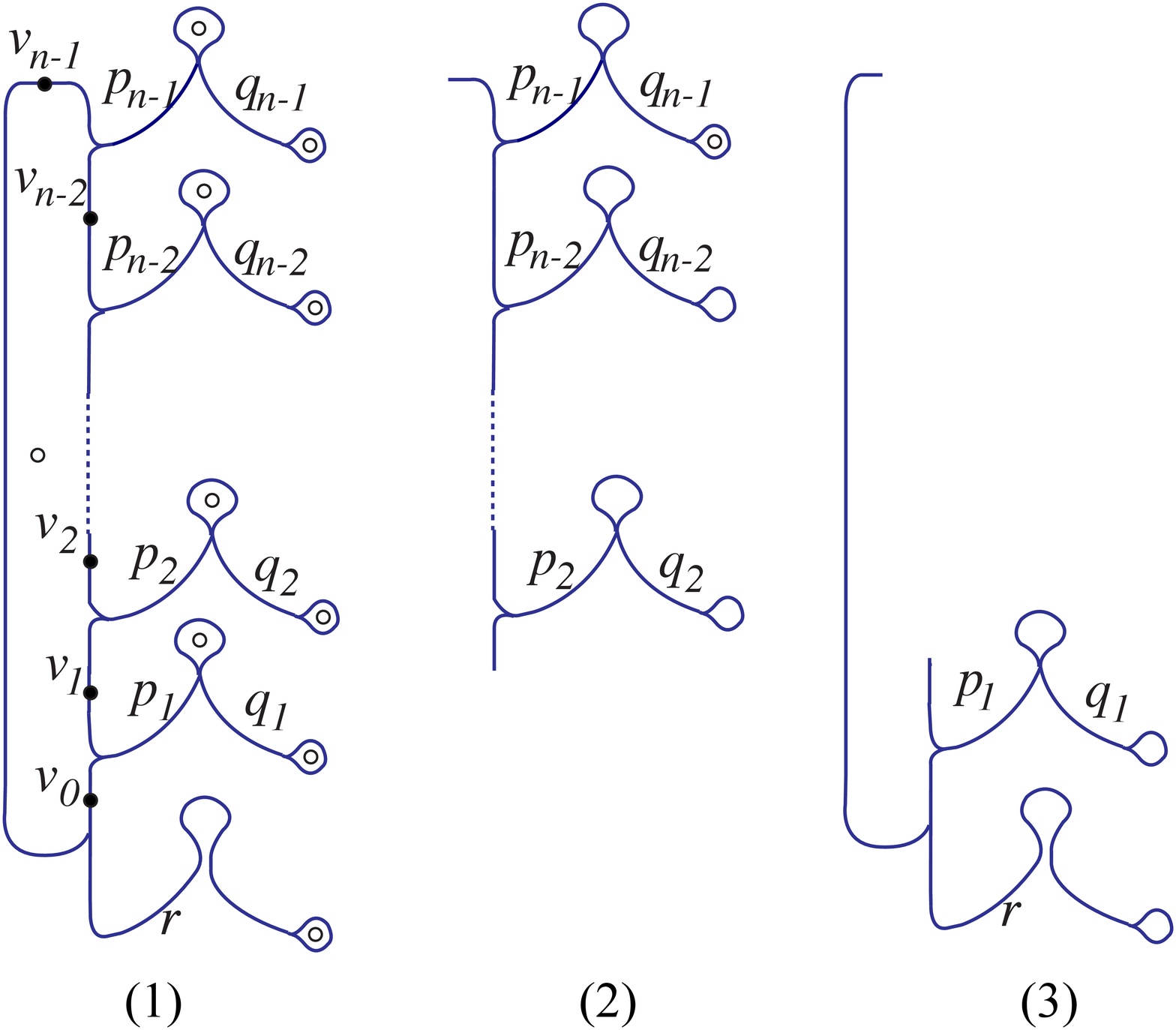}
\caption{(1) Points $v_0, v_1, \cdots, v_{n-1}$. 
(2) $\tau(2,n-1) \subset \mathcal{N}(p_2 q_2 \cdots p_{n-1} q_{n-1})$. 
(3) $\tau(1) \subset \mathcal{N}(r p_1 q_1)$.} 
\label{fig_subtrain}
\end{figure}

\begin{figure}[t]
\centering
\includegraphics[width=3.5in]{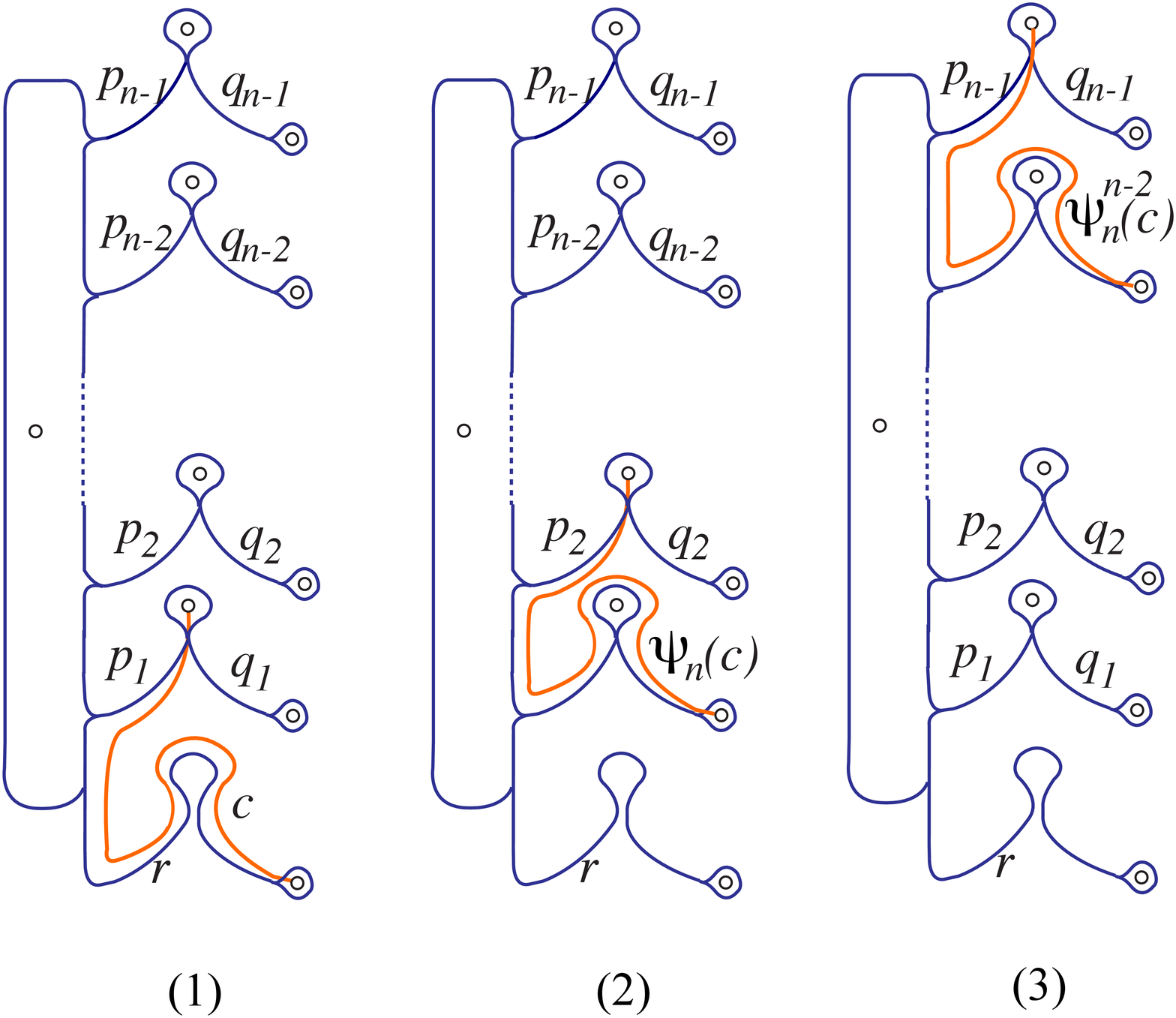}
\caption{(1) $c \subset \mathcal{N}(r p_1 q_1)$. 
(2) $\psi_n(c) \subset \mathcal{N}(p_1 q_1 p_2 q_2)$. 
(3) $\psi_n^{n-2}(c) \subset \mathcal{N} (p_{n-2} q_{n-2} p_{n-1} q_{n-1})$.} 
\label{fig_arc_curve}
\end{figure}

\begin{proof}
We first prove the latter upper bound. 
We assume $n \ge 4$. 
 Let $\mathcal{N}(\tau_n) \subset S_{0,2n+1}$ be a {\it fibered neighborhood} of $\tau_n$ 
 (see \cite[page~360]{PapadopoulosPenner87} for the definition) 
 equipped with a retraction $\mathcal{N}(\tau_n) \searrow \tau$. 
 For a connected subset $\tau' \subset \tau_n$, 
 we define a {\it fibered neighborhood}  $\mathcal{N}(\tau')$ of $\tau'$  as follows. 
 $$\mathcal{N}(\tau') = \mathcal{N}(\tau_n) \cap U(\tau'),$$
 where $U(\tau')$ is a small neighborhood of $\tau'$  in the $2$-sphere $S^2$. 
We take $n$ points $v_0, v_1, v_2, \cdots, v_{n-1} \subset \tau_n$, 
each of which lies on an edge of  the $(n-1)$-gon, 
see  Figure~\ref{fig_subtrain}(1).  
For $1 \le i < j \le n-1$, 
let $\tau(i,j)$ be the connected component of $\tau_n \setminus \{v_{i-1}, v_j\}$ 
containing $p_i, q_i, p_{i+1}, q_{i+1}, \cdots, p_{j}, q_{j}$
(see Figure~\ref{fig_subtrain}(2)). 
We consider its fibered neighborhood $\mathcal{N}(\tau(i,j))$, and we set 
$$\mathcal{N}(p_i q_i p_{i+1} q_{i+1} \cdots p_{j} q_{j}) = \mathcal{N}(\tau(i,j)).$$
For $1 \le j \le n-2$, 
let $\tau(j)$ be the connected component of $\tau_n \setminus \{v_j, v_{n-1}\}$ 
containing $r, p_1, q_1, \cdots, p_j, q_j$
(see Figure~\ref{fig_subtrain}(3)). 
Let 
$$\mathcal{N}(r p_1 q_1 \cdots p_j q_j) = \mathcal{N}(\tau(j)).$$
The notation $\mathcal{N}(r p_1 q_1 \cdots p_j q_j)$ tells a property that 
it contains  $r, p_1, q_1, \cdots, p_j, q_j$. 
The same thing holds for $\mathcal{N}(p_i q_i p_{i+1} q_{i+1} \cdots p_{j} q_{j})$.

We take an essential arc $c$ connecting the two punctures as in Figure~\ref{fig_arc_curve}(1). 
Then $c$ is carried by $\tau_n$. 
Notice that if $i \ge 2$, then 
$\mathcal{N}(p_i q_i p_{i+1} q_{i+1} \cdots p_{n-1} q_{n-1})$ is disjoint from $c$. 
Since $c \subset \mathcal{N}(r p_1 q_1)$, we have 
\begin{eqnarray*}
\psi_n(c) &\subset& \mathcal{N}(p_1 q_1 p_2 q_2), 
\\
\psi_n^2(c) &\subset& \mathcal{N}(p_2 q_2 p_3 q_3), 
\\
&\vdots&
\\
\psi_n^{1+ (n-3)}(c)= \psi_n^{n-2}(c) &\subset& \mathcal{N} (p_{n-2} q_{n-2} p_{n-1} q_{n-1})
\end{eqnarray*}
(see Figures~\ref{fig_digraph} and \ref{fig_arc_curve}).
Observe that  $\psi_n^2 (\psi_n^{n-2}(c)) = \psi_n^n(c) \subset \mathcal{N}(r p_1 q_1 p_2 q_2)$. 
We have 
\begin{eqnarray*}
\psi_n^{n+1} (c) &\subset& \mathcal{N}(p_1 q_1 p_2 q_2 p_3 q_3), 
\\
&\vdots&
\\
\psi_n^{(n+1)+ (n-4)}(c)= 
\psi_n^{2n-3} (c)  &\subset& \mathcal{N} (p_{n-3} q_{n-3} p_{n-2} q_{n-2} p_{n-1} q_{n-1}).
\end{eqnarray*}
In the same manner, for $2 \le k \le n-2$, we have
$$\psi_n^{(k-1)n-k}(c) \subset \mathcal{N}(p_{n-k} q_{n-k} \cdots p_{n-1} q_{n-1}).$$
When $k=n-2$, 
$$\psi_n^{(n-3)n- (n-2)}(c) = \psi_n^{n^2-4n+2}(c) \subset \mathcal{N} (p_2 q_2 \cdots p_{n-1} q_{n-1}).$$
Hence clearly we have 
$$d_{\mathcal{AC}}(c, \psi_n^{n^2-4n+2}(c))=1.$$
If we consider a regular neighborhood of $c$ in $S^2$, 
then we obtain an essential simple closed curve $\alpha$ in $ S_{0,2n+1}$ 
as the boundary of the neighborhood in question.  
Notice that  $\alpha$ is also carried by $\tau_n$ and $\alpha  \subset \mathcal{N}(r p_1 q_1)$. 
The above argument shows that $\psi_n^{n^2-4n+2}(\alpha) \subset \mathcal{N}(p_2 q_2  \cdots p_{n-1} q_{n-1})$ and 
$\alpha $ is disjoint from $\psi_n^{n^2-4n+2}(\alpha)$. 
Recall that $\psi_n$ is defined on $R_n= D_{2n}$. 
This together with Lemma~\ref{lemma:distance} implies that 
$$L_{\mathcal{C}}( \mathrm{Mod}(D_{2n})) \le \frac{1}{n^2-4n+2}.$$

To show the former upper bound of 
$ L_{\mathcal{C}}(\mathrm{Mod}(D_{2n-1})$ in the claim, 
we fill the  puncture  in the $(n-1)$-gon of $S_{0,2n+1} \setminus \tau_n$. 
The assumption $n-1 \ge 3$ ensures that  
$\tau_n$ extends to a train track $\overline{\tau}_n$ in $S_{0,2n}$ and 
$\psi_n: S_{0,2n+1} \rightarrow S_{0,2n+1} $ extends to 
$\overline{\psi}_n: S_{0,2n} \rightarrow S_{0,2n}$ which is still pseudo-Anosov. 
In particular 
$\overline{\psi}_n$ maps the puncture $p_{\infty}$ to itself.  
We can think of $\overline{\psi}_n: S_{0,2n} \rightarrow S_{0,2n}$ 
as an element of $\mathrm{Mod}(D_{2n-1})$. 
The train track representative 
$\mathfrak{p}_n: \tau_n \rightarrow \tau_n$ also extends to a train track representative 
$\overline{\mathfrak{p}}_n: \overline{\tau}_n \rightarrow \overline{\tau}_n$ of $\overline{\psi}_n: S_{0,2n} \rightarrow S_{0,2n}$. 
All non-loop edges of $\overline{\tau}_n$ are coming from those of $\tau_n$, and hence 
the directed graph $\overline{\Gamma}_n$ for $\overline{\mathfrak{p}}_n: \overline{\tau}_n \rightarrow \overline{\tau}_n$ 
is the same as $\Gamma_n$ for $\mathfrak{p}_n; \tau_n \rightarrow \tau_n$. 
For the arc $\overline{c}$ and the simple closed curve $\overline{\alpha}$  in $S_{0,2n}$ 
coming from $c$ and $\alpha$ in $S_{0,2n+1}$, respectively,
the above argument tells us  that 
\begin{equation}
\label{equation_disjoint}
d_{\mathcal{AC}} (\overline{c}, \ \overline{\psi}_n^{n^2-4n+2}(\overline{c}))=1 \hspace{2mm} \mbox{and}\hspace{2mm}
d_{\mathcal{C}} (\overline{\alpha},\ \overline{\psi}_n^{n^2-4n+2}(\overline{\alpha}))=1.
\end{equation}
The latter equality in (\ref{equation_disjoint})  with Lemma~\ref{lemma:distance} gives  the desired upper bound. 
\end{proof}

We are now ready to prove Theorem~\ref{prop:upperbound}. 

\begin{proof}[Proof of Theorem~\ref{prop:upperbound}]
By Lemma~\ref{lem_lift} together with either of the equalities for $\overline{\psi}_n: S_{0,2n} \rightarrow S_{0,2n}$
in (\ref{equation_disjoint}),  
we have 
$L_{\mathcal{C}} (\mathcal{H}(S_{n-1})) \le \frac{1}{n^2-4n+2}$ for $n \ge 4$. 
Thus for $g \ge 3$, 
$$L_{\mathcal{C}}(\mathcal{H}(S_g)) \le \frac{1}{(g+1)^2-4(g+1)+2} = \frac{1}{g^2-2g-1}.$$
\end{proof}

\section{Proof of Theorem~\ref{thm:hyphandle}}

\begin{figure}[t]
\centering
\includegraphics[width=4.5in]{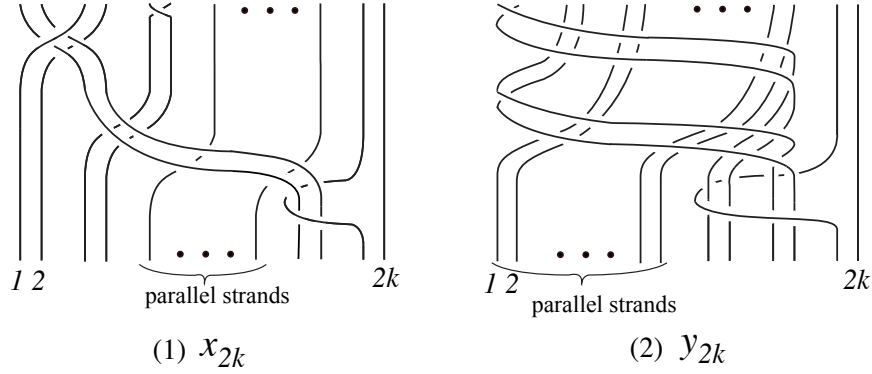}
\caption{(1) $x_{2k} \in SW_{2k}$. (2) $y_{2k} \in SW_{2k}$.} 
\label{fig_z}
\end{figure}

In this section, we finally prove Theorem~\ref{thm:hyphandle}. 

\begin{proof}[Proof of Theorem~\ref{thm:hyphandle}]
The proof is separated into two cases, depending on the parity of the genera. 
First of all we introduce spherical braids $x_{2k}, y_{2k} \in SB_{2k}$ for $k \ge 5$ 
as shown in Figure~\ref{fig_z}. 
It is straightforward to see that they are elements of $SW_{2k}$. 
We define $w_{2k} \in SW_{2k}$ for each $k \ge 5$ as follows. 
\begin{eqnarray*}
w_{4n+8}&=& x_{4n+8}(y_{4n+8})^n \hspace{1.2cm}\mbox{if}\ \hspace{2mm} 2k= 4n+8 \hspace{2mm} \mbox{for some}\ n \ge 1, 
\\
w_{4n+10}&=& (x_{4n+10})^2 (y_{4n+10})^n \hspace{0.5cm}\mbox{if}\ \hspace{2mm} 2k= 4n+10 \hspace{2mm} \mbox{for some}\ n \ge 0.
\end{eqnarray*}
Consider an element  in the Hilden group $SH_{2k}$  corresponding to  $w_{2k}$ 
(see Section~\ref{subsection_subgroup}) 
and its mapping torus $M_{w_{2k}}$.  
In \cite{HiroseKin17} it is shown that 
when $2k= 4n+8$ for $n \ge 1$,  
 $M_{w_{2k}}$ is homeomorphic to the mapping torus $M_w$ of the  element  in $SH_6$ 
 corresponding to the pseudo-Anosov braid $w \in SW_6$ (see Figure~\ref{fig_wicket}(3)). 
In other words,
$M_w$ is hyperbolic and it has a fiber $S_{0, 2k}$ with pseudo-Anosov monodromy $w_{2k}$ when $2k= 4n+8$. 
We claim that a sequence of fibers $(S_{0, 4n+10}, w_{4n+10})$ of $M_w$ comes from a fibered 3-manifold
as in Theorem~\ref{thm:main}. 
More precisely, 
if we remove the $6$th strand of $w$, then we obtain a spherical braid with $5$ strands. 
Regarding such a braid as the one on the disk, 
we have a $5$-braid, say $\psi \in B_5$. 
Clearly $M_{\psi}$ is homeomorphic to $M_w$. 
We consider a fiber $S= D_5$ with monodromy $\psi$ of the mapping torus $M_{\psi} \simeq M_w$. 
Since $\psi_*$ maps the generator $t_5$ to itself (see the $5$th strand of the braid $w$ in Figure~\ref{fig_wicket}(3)), 
the cohomology class $\xi_0 \in H^1(S; {\Bbb Z})$ which is dual to the proper arc $c= c_5$ is fixed by $\psi$. 
Let $\widetilde{S}$ be the ${\Bbb Z}$-cover of $S$  corresponding to $\xi_0$. 
We consider the canonical lift $\widetilde{\psi}: \widetilde{S} \rightarrow \widetilde{S}$ of $\psi$. 
Then $R_n = \widetilde{S}/\langle h^n\widetilde{\psi} \rangle$ is a fiber of $M_{\psi}$ with monodromy $\psi_n$ for $n$ large. 
In this case, $R_n$ is a sphere with $4n+8$ punctures, and  
we find that the monodromy $\psi_n$ is given by the braid $w_{4n+8} \in SW_{4n+8}$ 
from the argument in \cite[Section~3]{HiroseKin17}.  
By the proof of Theorem~\ref{thm:main}, there exist  $\alpha \in \mathcal{AC}(R_n)^{0}$ and $m \asymp n^2$ such that 
$d_{\mathcal{AC}}(\alpha, (\psi_n)^m (\alpha)) = 1$. 
Notice that a lift $\widehat{\psi}= \widehat{\psi_n}$ of $\psi_n$ 
under the map $q$ is an element of  $\mathcal{H}({\Bbb H}_{2n+3})$ (see Section~\ref{subsection_subgroup}). 
By Lemma~\ref{lem_lift}, 
% $\widehat{\psi} \in \mathcal{H}({\Bbb H}_{2n+3})$ satisfies 
$\ell_{\mathcal{C}} (\widehat{\psi}) \le 1/m$, which implies
$\ell_{\mathcal{C}} (\widehat{\psi}) \le C/n^2$ for some constant $C>0$. 
Thus we have $L_{\mathcal{C}}(\mathcal{H}({\Bbb H}_{2n+3})) \le  C/n^2$ 
in the case of the odd genus.

To obtain the upper bound $L_{\mathcal{C}}(\mathcal{H}({\Bbb H}_{2n+4})) \le  C'/n^2$ for some $C'>0$ 
in the case of the even genus, 
we take the second power $\psi^2 \in B_5$ of the above $\psi$ and we set $\phi= \psi^2$. 
We consider a fiber $S= D_5$ with monodromy $\phi$ in the mapping torus $M_{\phi}$. 
Note that  $\phi$ fixes the same $\xi_0 \in H^1(S; {\Bbb Z})$. 
Let $\widetilde{S}$ be the ${\Bbb Z}$-cover over $S$ as before and let 
$\widetilde{\phi}= (\widetilde{\psi})^2: \widetilde{S} \rightarrow \widetilde{S}$ which is the canonical lift of $\phi$. 
Now we apply Theorem~\ref{thm:main} for the fiber $(S, \phi)$ of $M_{\phi}$ 
together with $\xi_0$. 
One sees that 
for $n$ large, $ \widetilde{S}/\langle h^n\widetilde{\phi} \rangle$ is a fiber of $M_{\phi}$ which is  the sphere with $4n+10$ punctures. 
 The same argument as  in \cite[Section~3]{HiroseKin17} 
tells us that the monodromy of the fiber $ \widetilde{S}/\langle h^n\widetilde{\phi} \rangle$ 
is described by the braid $w_{4n+10} \in SW_{4n+10}$. 
As in the case of the odd genus, 
we obtain the desired upper bound of $L_{\mathcal{C}}(\mathcal{H}({\Bbb H}_{2n+4}))$. 
This  completes the proof. 
\end{proof}

\vspace{3em}
\bibliographystyle{alpha} 
\bibliography{trans}

\begin{thebibliography}{HPW15}

\bibitem[BH71]{BirmanHilden71}
Joan~S. Birman and Hugh~M. Hilden.
\newblock On the mapping class groups of closed surfaces as covering spaces.
\newblock pages 81--115. Ann. of Math. Studies, No. 66, 1971.

\bibitem[BH95]{BestvinaHandel95}
M.~Bestvina and M.~Handel.
\newblock Train-tracks for surface homeomorphisms.
\newblock {\em Topology}, 34(1):109--140, 1995.

\bibitem[Bow08]{Bowditch08}
Brian~H. Bowditch.
\newblock Tight geodesics in the curve complex.
\newblock {\em Invent. Math.}, 171(2):281--300, 2008.

\bibitem[FM12]{FarbMargalit12}
Benson Farb and Dan Margalit.
\newblock {\em A primer on mapping class groups}, volume~49 of {\em Princeton
  Mathematical Series}.
\newblock Princeton University Press, Princeton, NJ, 2012.

\bibitem[GT11]{GadreTsai11}
Vaibhav Gadre and Chia-Yen Tsai.
\newblock Minimal pseudo-{A}nosov translation lengths on the complex of curves.
\newblock {\em Geom. Topol.}, 15(3):1297--1312, 2011.

\bibitem[HK17]{HiroseKin17}
Susumu Hirose and Eiko Kin.
\newblock The asymptotic behavior of the minimal pseudo-{A}nosov dilatations in
  the hyperelliptic handlebody groups.
\newblock {\em Q. J. Math.}, 68(3):1035--1069, 2017.

\bibitem[HPW15]{HenselPrzytyckiWebb15}
Sebastian Hensel, Piotr Przytycki, and Richard C.~H. Webb.
\newblock 1-slim triangles and uniform hyperbolicity for arc graphs and curve
  graphs.
\newblock {\em J. Eur. Math. Soc. (JEMS)}, 17(4):755--762, 2015.

\bibitem[Kin15]{Kin15}
Eiko Kin.
\newblock Dynamics of the monodromies of the fibrations on the magic
  3-manifold.
\newblock {\em New York J. Math.}, 21:547--599, 2015.

\bibitem[KR]{KinRolfsen16}
Eiko Kin and Dale Rolfsen.
\newblock Braids, orderings, and minimal volume cusped hyperbolic
  $3$-manifolds.
\newblock Preprint is available at arXiv:1610.03241.

\bibitem[McM00]{McMullen00}
Curtis~T. McMullen.
\newblock Polynomial invariants for fibered 3-manifolds and {T}eichm\"uller
  geodesics for foliations.
\newblock {\em Ann. Sci. \'Ecole Norm. Sup. (4)}, 33(4):519--560, 2000.

\bibitem[MM99]{MasurMinsky99}
Howard~A. Masur and Yair~N. Minsky.
\newblock Geometry of the complex of curves. {I}. {H}yperbolicity.
\newblock {\em Invent. Math.}, 138(1):103--149, 1999.

\bibitem[MM00]{MasurMinsky00}
H.~A. Masur and Y.~N. Minsky.
\newblock Geometry of the complex of curves. {II}. {H}ierarchical structure.
\newblock {\em Geom. Funct. Anal.}, 10(4):902--974, 2000.

\bibitem[Pen91]{Penner91}
R.~C. Penner.
\newblock Bounds on least dilatations.
\newblock {\em Proc. Amer. Math. Soc.}, 113(2):443--450, 1991.

\bibitem[PP87]{PapadopoulosPenner87}
Athanase Papadopoulos and Robert~C. Penner.
\newblock A characterization of pseudo-{A}nosov foliations.
\newblock {\em Pacific J. Math.}, 130(2):359--377, 1987.

\bibitem[Val14]{Valdivia14}
Aaron~D. Valdivia.
\newblock Asymptotic translation length in the curve complex.
\newblock {\em New York J. Math.}, 20:989--999, 2014.

\end{thebibliography}

\Addresses

\end{document}